\documentclass{amsart}[12]
\usepackage[utf8]{inputenc}
\usepackage{hyperref}
\hypersetup{hidelinks}
\usepackage{url}
\usepackage{quiver}
\usepackage{amsmath}
\usepackage{amsfonts}
\usepackage{amssymb}
\usepackage{mathtools}
\usepackage{tikz}
\usepackage{tikz-cd}
\usepackage{stmaryrd}
\usepackage{amsmath,calligra,mathrsfs}
\DeclareMathOperator{\HHom}{\mathscr{H}\text{\kern -3pt {\calligra\large om}}\,}
\DeclareMathOperator{\Spec}{Spec}
\DeclareMathOperator{\Hom}{Hom}
\DeclareMathOperator{\ord}{ord}
\DeclareMathOperator{\Ext}{Ext}
\DeclareMathOperator{\End}{End}

\DeclareMathOperator{\Gm}{\mathbb{G}_m}
\DeclareMathOperator{\BGm}{B \mathbb{G}_m}
\DeclareMathOperator{\Br}{Br}
\DeclareMathOperator{\D}{D}
\DeclareMathOperator{\X}{\mathscr{X}}
\DeclareMathOperator{\W}{\mathscr{W}}
\DeclareMathOperator{\s}{\mathscr{S}}
\DeclareMathOperator{\A}{\mathscr{A}}
\DeclareMathOperator{\B}{\mathscr{B}}
\DeclareMathOperator{\E}{\mathscr{E}}
\DeclareMathOperator{\N}{\mathcal{N}}
\DeclareMathOperator{\M}{\mathcal{M}}
\DeclareMathOperator{\LL}{\mathscr{L}}
\DeclareMathOperator{\Y}{\mathscr{Y}}
\DeclareMathOperator{\F}{\mathscr{F}}
\DeclareMathOperator{\G}{\mathscr{G}}
\DeclareMathOperator{\Z}{\mathscr{Z}}
\DeclareMathOperator{\FF}{F}

\DeclareMathOperator{\GG}{G}
\DeclareMathOperator{\Aut}{Aut}
\DeclareMathOperator{\AAut}{\mathscr{A}\text{\kern -3pt {\calligra\large ut}}\,}

\DeclareMathOperator{\Char}{char}
\DeclareMathOperator{\Pic}{Pic}
\DeclareMathOperator{\Autt}{\mathscr{A}ut}
\DeclareMathOperator{\im}{Im}
\DeclareMathOperator{\rk}{rk}
\DeclareMathOperator{\codim}{codim}
\DeclareMathOperator{\Hilb}{Hilb}
\DeclareMathOperator{\Coh}{Coh}

\usepackage{graphicx}
\usepackage{amsmath}

\usepackage{pgfplots}
\usepackage{tikz}
\usepackage{graphicx}
\graphicspath{ {./images/} }
\usepackage{amsmath}
\usepackage{amsthm}
\usepackage{amssymb}
\usepackage{mathabx}
\usepackage{mathrsfs}
\usepackage{amsfonts}
\usepackage{enumitem}
\usepackage{amssymb}

\usepackage[backend=biber,
style=alphabetic,
isbn=false,
url=false,
sorting=nyt]{biblatex}
\usepackage{bm}
\graphicspath{ {./images/} }
\usetikzlibrary{shapes}
\usepgfplotslibrary{polar}
\usetikzlibrary{decorations.markings}
\usetikzlibrary{backgrounds}
\pgfplotsset{every axis/.append style={
                    axis x line=middle,    % put the x axis in the middle
                    axis y line=middle,    % put the y axis in the middle
                    axis line style={<->,color=blue}, % arrows on the axis
                    xlabel={$x$},          % default put x on x-axis
                    ylabel={$y$},          % default put y on y-axis
            }}
\usepackage[utf8]{inputenc}

\newtheorem{theorem}{Theorem}[subsection]
\newtheorem{lemma}[theorem]{Lemma}

\newtheorem{proposition}[theorem]{Proposition}
\newtheorem{corollary}[theorem]{Corollary}

\newtheorem{definition}[theorem]{Definition}
\newtheorem{example}[theorem]{Example}
\newtheorem{remark}[theorem]{Remark}
 
\usepackage[margin=1in]{geometry}

\title{Point Objects and Derived Equivalences of Twisted Derived Categories of Abelian Varieties}
\author{Ruoxi Li}

\address[R. Li]{Department of Mathematics\\
Univeristy of California, Berkeley\\
Berkeley, CA 94720\\
U.S.A.}
\email{ruoxi\_li@berkeley.edu}

\begin{document}

\begin{abstract}
    We study the notion of $1$-twisted semi-homogeneous vector bundles on $\Gm$-gerbes over abelian varieties, and classify point objects in the twisted derived categories of abelian varieties. As an application, we classify the twisted Fourier-Mukai partners of abelian varieties.
\end{abstract}
\maketitle

\section{Introduction}
Let $k$ be a field of characteristic 0, let $X/k$ be a torsor under an abelian variety $A/k$ of dimension $d$, and let $p \colon \X \rightarrow X$ be the $\Gm$-gerbe over $X$ corresponding to a class $\alpha \in \Br(X)$ with period $n$. Fix a $\mu_n$-gerbe $\X_n$ and an isomorphism $\iota: \X \cong \X_{n} \times^{\mu_n} \mathbb{G}_m$. With this fixed $\mu_n$ lift, we then can define the slope of vector bundles on $\X_n$, and hence on $\X$ via the equivalence \[\{\text{1-twisted coherent sheaves on } \X \} \xrightarrow{\iota^*} \{\text{1-twisted coherent sheaves on } \X_n \}.\]

\begin{definition}
    We define the slope of a 1-twisted line bundle $\LL$ on $\X_n$ to be $\delta(\LL) = \frac{\det{\LL}^{\otimes n}}{n} \in NS(X)_{\mathbb{Q}}$.

    We define the slope of a 1-twisted vector bundle $\E$ on $\X_n$ to be $\delta{(\E)} = \frac{\delta(\det(\E))}{\rk(\E)} \in NS(X)_{\mathbb{Q}}$.
\end{definition}

\begin{remark}
    This definition agrees with the definition of the slope of twisted sheaves in \cite{krashen2008index}.

    In this paper, we define the slope of vector bundles $\E$ on $\X$ as $\delta( \iota^*\E)$
\end{remark}

Let $A^{\vee}$ denote the dual abelian variety of $A$ and let $\D(\X)^{(1)}$ denote the bounded derived category of 1-twisted coherent sheaves on $\X$; more concretely, it is the category of objects $\FF \in \D(\X)$ for which the action of the inertia group $\Gm$ on the cohomology sheaves is the standard action. The category $\D(\X)^{(1)}$ is often denoted $\D(X, \alpha )$ and can be described in terms of Azumaya algebras, see \cite{caldararu2000derived}. In this paper, we classify twisted Fourier-Mukai partners of abelian varieties and generalize several results in \cite{mukai1978semi} and \cite{de2022point} to the twisted setting. The key to our work in this paper is the following notion: 
\begin{definition}
    If $k$ is algebraically closed, a 1-twisted vector bundle $\E$ on $\X$ is called $\textbf{semi-homogeneous}$ if for every $\sigma \in \Aut^0_{\X}(k)$ there exists a 0-twisted line bundle $\LL$ such that \[\sigma^*\E \cong \E \otimes \LL\] 
    For general $k$, we call a 1-twisted vector bundle semi-homogeneous if its base change to an algebraic closure $\bar{k}$ of $k$ is semi-homogeneous.
\end{definition}

The group space $\Aut^0_{\X}$ is defined as following.

Let  $\AAut_{\X}$ be the fibered category over $k$ which to any $k$-scheme $S$ associates the groupoid of isomorphisms $\X \rightarrow \X$ inducing the identity on the stabilizer group $\Gm$ (ie. isomorphism of $\Gm$-gerbes). Let $\AAut^0_{\X}$ be its neutral connected component and let $\Aut^0_{\X}$ be the coarse space of $\AAut^0_{\X}$. By \cite[Theorem 1.1]{olsson2025twisted}, the map $\AAut^0_{\X} \rightarrow \Aut^0_{\X}$ is a $\Gm$-gerbe and in our setting ($\X$ is a $\Gm$-gerbe over an abelian variety $X$,) $\Aut^{0}_{\X}$ is an abelian variety.

\begin{remark}
    We can think of $\sigma \in \Aut^0_{\X}(k)$ as the set of isomorphism classes of automorphisms of $\X$ since $\Aut^0_{\X}$ is the coarse space of $\AAut^0_{\X}$.
\end{remark}

The main theorems of this paper are listed in the following three subsections. We eventually classify the twisted Fourier-Mukai partners of abelian varieties by classifying point objects in the twisted derived category $\D(X, \alpha)$, which reveals the geometry encoded in the twisted derived categories.

One of the key ingredients in studying point objects on abelian varieties is the nice properties of semi-homogeneous vector bundles on abelian varieties proved by Mukai. The idea of semi-homogeneous vector bundles on the gerbes is generalizing the classical results on abelian varieties \cite{mukai1978semi}, which can also be viewed as the case on the trivial $\Gm$-gerbe over an abelian variety. 

\begin{remark}
    We expect the $\Char{k} = 0$ assumption can be removed.
\end{remark}

\subsection{Basic Properties of Semi-Homogeneous Vector Bundles on Gerbes over Abelian Varieties}

 On $\Gm$-gerbes, we still get basic properties of semi-homogeneous vector bundles, proven in the untwisted case in \cite{mukai1978semi}. Over an algebraically closed field, we see that semi-homogeneous vector bundles still behave nicely under pullback and pushforward along isogenies. More importantly, we have the following structural results:
 %\begin{enumerate}
%    \item for any semi-homogeneous vector bundle $\E$ we can still find a map  such that $E$ is a pushforward of a line bundle;
%     \item every semi-homogeneous vector bundle has a decomposition into filtrations whose successive quotients are simple semi-homogeneous vector bundles.
 %\end{enumerate}
 
\begin{theorem}[Theorem \ref{linebundle}]

Assume $\bar{k} = k$. If $\E$ is simple (recall that $\E$ is called simple if $\Hom(\E,\E) = k$) semi-homogeneous, then there exists an isogeny $f : X' \rightarrow X$ such that $\X' := \X \times_X X' \cong \BGm_{X'}$ and $\E \cong \tilde{f}_* \LL$ for some 1 twisted line bundle on $\X'$ where $\tilde{f}$ is the base change of $f$ under $p : \X \rightarrow X$.
\end{theorem}

\begin{theorem}[Theorem \ref{summand}]
Assume $\bar{k} = k$. A 1-twisted vector bundle $\E$ on $\X$ is semi-homogeneous if and only if $\E \cong \bigoplus_{\E_i}U_{\X, \E_i}$ where $U_{\X, \E_i}$ has a filtration whose successive quotients are isomorphic to a fixed $\E_i$ and $\E_i$'s are non-isomorphic simple semi-homogeneous vector bundles such that all $\delta(\E_i):= \frac{\det(\E_i)}{rk(\E_i)}$ are equal in $NS(X) \otimes \mathbb{Q}$, and only finitely many $U_{\X, \E_i} \neq 0$.
\end{theorem}

\begin{remark}
    To define $\delta(\E)$, we use the fact that the category of $1$-twisted vector bundles on $\X$ is equivalent to the $1$-twisted vector bundles on $\X_n$. $\delta(\E)$ is well-defined for $1$-twisted vector bundles $\E$ on $\X_{n}$. Indeed, $\det(\E)$ is a $0$-twisted line bundle on $\X_{n}$, which is the pullback of a unique vector bundle on the coarse space $X$ to $\X_n$.

    Note that this notion is only well-defined after we fix the $\mu_n$-gerbe $\X_n$ and the isomorphism $\iota$.
\end{remark}
\subsection{Point Objects in Twisted Derived Categories over Abelian Varieties}
Using the results of semi-homogeneous vector bundles, we can analyze the behavior of semi-homogeneous complexes on $\Gm$-gerbes over abelian varieties, which leads to a classification of point objects. The basic idea is similar to the non-twisted case in \cite{de2022point}. This gives us our second main result, generalizing the classification of point objects on abelian varieties by de Jong and Olsson, \cite[Theorem 1]{de2022point}.

\begin{theorem}[Section 5]
\label{point}
    Let $\FF \in \D(X, \alpha)$ be an object that satisfies the following:

    \begin{enumerate}
        \item $\Ext^i(\FF, \FF) = 0$ for all $i < 0$;
        \item The k-vector space $\Ext^0(\FF, \FF)$ has dimension 1;
        \item the k-vector space $\Ext^1(\FF, \FF)$ has dimension $\leq d$.
    \end{enumerate}

    Then \[\FF \cong \tilde{i}_* \E[r]\] where $r$ is an integer, $i \colon Z \xhookrightarrow{} X$ is a torsor under a sub-abelian variety $H \subset A$, and $\tilde{i}$ is the base change of $i$ under $p$, and $\E$ is a simple semi-homogeneous vector bundle on $\mathscr{Z} := \X \times_X Z$.
    
    Conversely, if $i \colon Z \xhookrightarrow{} X$ is a torsor under a sub-abelian variety $H \subset A$, $\tilde{i}$ is the base change of $i$ under $p$, and $\E$ is a simple semi-homogeneous vector bundle on $\mathscr{Z} := \X \times_X Z$, then $\FF = \tilde{i}_* \E$ satisfies the conditions above.
\end{theorem}

\begin{definition}
    Any $\FF \in \D(\X)^{(1)}$ that satisfies the conditions in Theorem \ref{point} is said to be a \textbf{point object}.
\end{definition}

\subsection{Twisted Fourier Mukai Partners with Abelian Varieties}
The main application of Theorem \ref{point} is the following. This theorem concretely classifies the twisted Fourier-Mukai partners of an abelian variety and the twists.

\begin{theorem}[Section 6]
    Let $X$ be a torsor under an abelian variety $A$, and $Y$ a smooth projective variety. Let $\alpha \in \Br(X), \beta \in \Br(Y)$. If $\D(X, \alpha) \cong \D(Y, \beta)$, then $Y \cong M_{(S, \delta)}$ (See section 6 for definition) for some subabelian variety $S \subset A$ and $\delta \in NS(S) \otimes_{\mathbb{Z}} \mathbb{Q}$, and $\beta \in \Br(Y)$ is the inverse of the class of the $\mathbb{G}_m$-gerbe $\M_{(S, \delta)} \to M_{(S, \delta)}$ (again, for more details, see section 6). In particular, $Y$ is also a torsor under an abelian variety of dimension $\dim{A}$.
    
    % in particular, $Y \cong SSH_Z^{\xi_Z}$ for some $Z \xrightarrow[]{} X$ a torsor under a sub-abelian variety of $A$.
\end{theorem}

\begin{remark}
    The non-twisted case is studied in \cite[Theorem 2.4]{lane2024semi} and \cite[proposition 3.3]{kurama2024fourier}.
\end{remark}

\subsection{Funding}
This work was partially funded by the Simons Collaboration on Perfection in Algebra, Geometry, and Topology. 

\subsection{Acknowledgment}

Thanks to Martin Olsson and Noah Olander for many helpful comments and conversations.

%%The idea of semi-homogeneous vector bundles on the gerbes is generalizing the classical results on abelian varieties \cite{mukai1978semi}, which can also be viewed as the case on the trivial $\mathbb{G}_m$-gerbe over an abelian variety.  It turns out that on the $\mathbb{G}_m$-gerbes, the good properties semi-homogeneous vector bundles  still hold true. More concretely, we see that semi-homogeneous vector bundles still behave nicely under pullback and pushforward along nice isogenies. More importantly, we have 1. for any semi-homogeneous vector bundle $E$ we can still find a map  such that $E$ is a pushforward of a line bundle; and 2. every semi-homogeneous vector bundle have a decomposition into filtrartions whose successive quotients are simple semi-homogeneous vector bundles.

\section{Representable Functors}

\subsection{Moduli of Complexes}
In this subsection, let $X \rightarrow S$ be a proper smooth morphism of finite presentation between Noetherian schemes. Let $\X \rightarrow X$ be a $\Gm$-gerbe corresponding to $\alpha \in \Br(X)$. Let $\mathscr{D}_{\X}$ be the fibered category which to any $T \rightarrow S$ associates the groupoid of objects $E \in \D(\X)^{(1)}$ which are relatively perfect over $T$ and such that $\Ext^i(E_s, E_s) = 0$ for all geometric points $s \rightarrow S$ and all $i <0$.

\begin{proposition}
\label{complexes}
    $\mathscr{D}_{\X}$ is an algebraic stack.
\end{proposition}

\begin{proof}
     By \cite[Corollary 6.5]{bergh2021decompositions}, for any $T/S$, $\D(\X_T)^{(1)}$ is a $X_T$ linear semi-orthogonal component of $\D(Y_T)$ for some $Y$ a Brauer Severi variety over $X$ (and hence $Y_T$ is a Brauer Severi variety over $X_T$) that corresponds to $\alpha^{-1} \in \Br(X)$,
    \[ \D(Y_T) \cong <\D(X_T), \D(X_T, \alpha^{-1}), \cdots, \D(X_T, \alpha)>\] 
    This decomposition is functorial, hence we have an embedding $\mathscr{D}_{\X} \hookrightarrow \mathscr{D}_{Y}$.
    
    Note that each component in the semi-orthogonal decomposition is finitely generated over $X_T$, and that the generators can be obtained by pullback from $X$ over $S$. Call these generators $G_i$ on $Y$. 

    By \cite[Theorem 4.2.1]{lieblich2005moduli} $\mathscr{D}_Y$ is an algebraic stack.

    Let $F$ be the universal object on $Y \times \mathscr{D}_Y$, then for any $T' \rightarrow \mathscr{D}_Y$ over $S$, $[F|_{T'}] \in \mathscr{D}_{\X}$ if and only if \[R q_* R \HHom (F|_{T'}, G_i|_{T'}) = 0\] for all $i$ where $q: Y \times S \rightarrow S$. Note that under our assumptions $Rq_* R \HHom(F, G_i)$ is a perfect complex that behaves well with respect to base change, whose vanishing locus is open.
    
    This implies that $\mathscr{D}_{\X} \hookrightarrow \mathscr{D}_{Y}$ is an open immersion. So we have $\mathscr{D}_{\X}$ is algebraic.
\end{proof}

\subsection{Discussion on $\Phi_{\X}(\E)$ When $\E$ is a Simple Vector Bundle}
Following Mukai, the following analysis will be key to our classification of vector bundles. The subgroup of $X \times \Pic^0_{X}$,
$\{(a,\mathcal{L}) \in X \times X^\vee| t_a^* E \cong E \otimes \mathcal{L}\}$,
plays an important role in the classical theory. Here we consider the analogous functor in our setting.
% \[\Phi_{\X}(\E) :=(T \rightarrow \{\sigma \in \Aut^0(\X_T)| \exists \{U_i \rightarrow T \} \text{ s.t. } \sigma_i^* \E_{(U_i)} \cong \E_{(U_i)}) \] where $\{U_i \rightarrow T\}$ is an etale cover of $T$. 

As discussed in the introduction, since $\X \rightarrow X$ is a $\Gm$-gerbe over an abelian variety, we know that $\AAut^0_{\X} \rightarrow \Aut^0_{\X}$ is a $\Gm$-gerbe over an abelian variety by \cite[Theorem 1.1]{olsson2025twisted}. 

Let $\widetilde{\Phi_{\X}}(\E)$ be the fibered category over $k$ which to any scheme $T$ associates the groupoid whose objects are $\{\sigma \in \AAut^0(\X_T)|$ there exists $\{U_i \rightarrow T\} \text{ an etale cover s.t. }\sigma_i^* \E_{U_i} \cong \E_{U_i}\}$, and whose morphisms are natural transformations between the objects (inherited from $\AAut^0_{\X}$). $\widetilde{\Phi_{\X}}(\E)$ is a substack of $\AAut^0_{\X}$. In the next subsection we show that when $\E$ is a simple vector bundle, $\widetilde{\Phi_{\X}}(\E)$ is an algebraic stack, and we define $\Phi_{\X}(\E)$ as the coarse space of $\widetilde{\Phi_{\X}}(\E)$.
 
There is a $k$-morphism $\widetilde{\Phi_{\X}}(\E) \rightarrow \AAut^0_{\X}$. Our strategy to show the representability of $\Phi_{\X}(\E)$ is to show that it is indeed a locally closed subscheme of $\Aut^0_{\X}$.

\begin{remark}
    When $k = \bar{k}$, $0: \Spec{k} \to X$ lifts to a map $\Spec{k} \to \X$. The restriction of the composition of universal morphism and projection to $\X$ $\sigma: \AAut^0_{\X} \times \X \to \AAut^0 \times \X \to \X$ becomes $\sigma_0: \AAut^0_{\X} \to \X$.
    Note that the following diagram is cartesian.
    \[\begin{tikzcd}
	{\AAut^0_{\X}} & \X \\
	{\Aut^0_{\X}} & X
	\arrow["{\sigma_0}", from=1-1, to=1-2]
	\arrow[from=1-1, to=2-1]
	\arrow[from=1-2, to=2-2]
	\arrow[from=2-1, to=2-2]
\end{tikzcd}\] where the bottom map is the map in \cite[Theorem 1.1]{olsson2025twisted}.
\end{remark}

\subsection{Proof of Representability When $\E$ is a Simple Vector Bundle}
This subsection shows that functors like $\widetilde{\Phi_{\X}}(\E)$ are a $\Gm$-gerbe over a scheme. In fact, we study the more general functor $\W$ defined below. 

\begin{definition}
    Let $\s \rightarrow S$ be a $\Gm$-gerbe over a $k$-scheme $S$. Denote $V := \X \times_k S$.
    Given $(1,0)$-twisted locally free sheaves $F$ and $G$ that are universally simple (i.e. $p_*\End(F_T) \cong \mathcal{O}_T$ and $p_*\End(G_T) \cong \mathcal{O}_T$ for any $T/k$ where the isomorphisms are compatible with base change) on $\X \times_k \s$, let $\mathscr{W}$ be the fibered category over $S$ which to any $S$-scheme $T$ to the groupoid of $t \in \s(T)$ such that $F_T \cong G_T$ etale locally on $T$.
\end{definition}

\begin{remark}
\label{restriction to 0}
    Note that $\widetilde{\Phi_{\X}}(\E)$ is the special case of $\W$ when $S = \AAut^0_{\X_T}$, $F= \tilde{\sigma}^*\E_{\AAut^0_{\X}}$, and $G = \E_{\AAut^0_{\X}}$ with $\tilde{\sigma}$ being the universal object on $\X \times \AAut^0_{\X}$.
\end{remark}

\begin{lemma}
    Let $f \colon X  \rightarrow k$ be a proper flat universally integral morphism, let $q_S \colon \s \rightarrow S$ be a $\mathbb{G}_m$-gerbe over $S$, $q_X \colon \X \rightarrow X$ be a $\Gm$-gerbe over $X$, hence $q \colon \X \times_k \s \rightarrow X \times_k S$ is a $\Gm \times \Gm$-gerbe over $X \times_k S$. Let $F$ and $G$ be $(1,0)$-twisted locally free sheaves on $\X \times \s$. There exists a coherent $\mathcal{O}_S$-module $A$ and an isomorphism of functors on quasi-coherent $\mathcal{O}_S$-modules $M$:
    \[\tilde{g}_*(\HHom_{\mathcal{O}_{\tilde{V}}} (F, G) \otimes_{\mathcal{O}_S} M) \cong \HHom_{\mathcal{O}_S}(A, M).\] where $g := f_S \circ q$.
\end{lemma}

\begin{proof}
    For any $(1,0)$-twisted $F$ and $G$, $\HHom_{\mathcal{O}_{\tilde{V}}}(F, G)$ is 0-twisted and hence descends to the coarse space $X \times S$.
    % \[\HHom_{\mathcal{O}_{\X \times \s}} (F, G) \cong \HHom_{\mathcal{O}_{\X \times \s}}(\mathcal{O}_{\X \times \s}, \HHom_{\mathcal{O}_{\X \times \s}}(F, G)) \cong q^*\HHom_{\mathcal{O}_{X \times S}}(\mathcal{O}_{X \times S}, \HHom_{\mathcal{O}_{\X \times \s}}(F, G)) \otimes_{\mathcal{O}_S} M)\]

    So \[g_*(\HHom_{\mathcal{O}_{\X \times \s}}(F, G) \otimes_{\mathcal{O}_S} M)  \cong {f_S}_*(\HHom_{\mathcal{O}_{\X \times \s}}(F, G)\otimes_{\mathcal{O}_S} M)\]
    Then the result follows from \cite[7.7.6]{EGAIII}.
\end{proof}

In particular, given any morphism $\alpha: T \rightarrow S$, letting $\M = \alpha_*\mathcal{O}_T$, we see that 
\[\alpha_*g_{T*}(\HHom_{\mathcal{O}_{\X \times \s \times_S T}}(F_T, G_T) \cong \HHom_{\mathcal{O}_S}(A, \alpha_*{\mathcal{O}_T}),\] 
\[Hom_{\mathcal{O}_{\X \times \s \times_S T}}(F_T, G_T) \cong Hom_{\mathcal{O}_S}(A, \alpha_*{\mathcal{O}_T})\] for some $\mathcal{O}_S$-module $A$. 

Let $Z$ be the scheme-theoretic support of $A$, and let $W$ be the set $\{s \in S| F_s \cong G_s\}$ (similarly as in \cite[Proposition 1.5]{mukai1978semi}, $W$ is a constructible set.) As topological spaces, $W \subset Z$. 

When $\G$ is $S$-simple, as in \cite[Proposition 1.7]{mukai1978semi}, we see that for every point $s \in W$, the isomorphism $F_s \cong G_s$ extends to an isomorphism in an open neighborhood, so $W$ is an open subset of $Z$, which gives $W$ a natural open subscheme structure of $Z$.

\begin{proposition}
    $\W \rightarrow W$ is a $\Gm$-gerbe.
\end{proposition}

\begin{proof}
    We first show that the map $\W \rightarrow S$ factors through $W$, that is, to show that for any $T \rightarrow \W$ over $S$, we have $T \rightarrow W$ over $S$ (ie. $\W$ is a fibered category over $W$.)

    Given any $\alpha: T \rightarrow \W$ (we may assume $\alpha$ is affine), we see that there exists an etale cover of $\{U_i \rightarrow T\}$ such that $F_{U_i} \cong G_{U_i}$. This implies that there exists a line bundle on $T$ such that $F_T \cong G_T \otimes_{\mathcal{O}_T} N$ on $\X \times \s \times_S T$. So we have \[\HHom_{\mathcal{O}_{\X \times \s \times_S T}}(F_T, G_T) \cong \HHom_{\mathcal{O}_{\X \times \s \times_S T}}(G_T, G_T) \otimes_{\mathcal{O}_T} N.\]
    This implies that we have an injection \[g_{T}^* N \hookrightarrow \HHom_{\mathcal{O}_{\X \times \s \times_S T}}(F_T, G_T).\] Applying $g_{T*}$ with the fact that $g_{T*} \mathcal{O}_{\X \times \s \times_S T} \cong \mathcal{O}_T$, we get \[N \hookrightarrow g_{T*}\HHom_{\mathcal{O}_{\X \times \s \times_S T}}(F_T, G_T).\]
    Since $\alpha$ is affine, \[\alpha_*N \hookrightarrow \alpha_*g_{T*}\HHom_{\mathcal{O}_{\X \times \s \times_S T}}(F_T, G_T).\]

    So the annihilator of $A$ annihilates $\alpha_* N$. Since $N$ is an invertible sheaf on $T$, we see that $T \rightarrow S$ factors through $Z$, it then factors through $W$ because $F_T$ and $G_T$ are locally isomorphic.

    Now we show that $\W$ is a indeed a $\Gm$-gerbe over $W$ by verifying the conditions in \cite[Definition 12.2.2]{olsson2023algebraic}. The fact that G2 and G3 hold is straightforward because $\s$ is a gerbe over $S$. To show G1, it suffices to show that there exists a cover of $W$, $\{W_i \rightarrow W\}$, such that $\W(W_i)$ is nonempty for every $i$. As we mentioned in the last proof, for every point $x \in W$, $F_x, \cong G_x$ extends to an isomorphism $F_{W_x} \cong G_{W_x}$ where $W_x$ is an open neighborhood of $x$. Ranging over all the topological points of $W$, $\{W_x\}_{x \in W}$ is the desired cover.     
\end{proof}

In particular, when $\E$ is simple, apply the proposition to $S = \Aut^0_{\X}$ and $\s = \AAut^0_{\X}$, we see that $\Phi_{\X}(\E)$ is a subgroup scheme of $\Aut^0(\X)$, and $\widetilde{\Phi_{\X}}{(\E)} \rightarrow \Phi_{\X}(\E)$ is a $\Gm$-gerbe. Let $\Phi^{00}_{\X}(\E)$ be the neutral connected component of $\Phi_{\X}(\E)$ and $\widetilde{\Phi^{00}_{\X}(\E)}$ be the neutral connected component of $\widetilde{\Phi_{\X}(\E)}$.

\begin{remark}
    $\Phi_{\X}(\E)$ defined here is the twisted version of what Mukai studied in \cite[Section 3]{mukai1978semi}. Indeed, if $\X \cong \BGm_{,X}$, our definition aligns with \cite[Definition 3.5]{mukai1978semi}. The slightly unusual-looking notation $\Phi^{00}_{\X}$ is to be consistent with \cite[Definition 3.10]{mukai1978semi}. 
\end{remark}

\subsection{$\Phi(\E)$ when $\E$ is not simple}
 In \cite{mukai1978semi}, $\Phi^{00}_{\BGm_{,X}}(\E)$ is defined also for vector bundles $\E$ that are not simple. In this subsection, we discuss the situation without the assumption that $\E$ is a simple vector bundle.

 Let $\E \in \D(\X)^{(1)}$, let $S = \Aut^0_{\X}$ and $\s = \AAut^0_{\X}$, $V = \X \times_k S$, so $V$ is a $\Gm$-gerbe over $X \times_k S$ and $\X \times_k \s$ is a $\Gm$-gerbe over $V$.

 Let $F= \tilde{\sigma}^*\E_{\AAut^0_{\X}}$, $G = \E_{\AAut^0_{\X}}$ with $\tilde{\sigma}$ being the universal object on $\X \times \AAut^0_{\X}$.
 
 Since $F$ and $G$ are $(1,0)$-twisted on $\X \times_k \s$, they can be viewed as objects in $\D(V)^{(1)}$. By Proposition \ref{complexes} and \cite[\href{https://stacks.math.columbia.edu/tag/045G}{Tag 045G}]{stacks-project}, we see that the functor $\underline{Isom}$ which sends $T/S$ to the set ${Isom}_{V_T}(F_T, G_T)$ is an algebraic space locally of finite presentation over $S$. Denote its structure morphism by $h: \underline{Isom} \rightarrow S$.

% Let $p_1, p_2, p_3, p_{12}, p_{13}, p_{23}$ be the various projections of $V := X \times X \times X^{\vee}$, and let $E$ be a complex in $\D(X) \cong \D(\BGm_{,X})^{(1)}$. Let $\FF = (p_{12}^*m^*(E)) \otimes p_{13}^*(\mathcal{P}^{-1})$, $\GG = p_1^*{(E)}$, where $\mathcal{P}$ is the Poincare line bundle.

 \begin{definition}
     For $\E \in \D(\X)^{(1)}$, 
     \[\Phi_{\X}^0{(\E)} = \{s \in S | F_{\X \times \{s\}} \cong G_{\X \times \{s\}} \}.\]
    We omit the subscript $\X$ when there's no ambiguity.
 \end{definition}

 \begin{remark}
     For $k$-points, we see that $\Phi^0(\E)(k)$ consists of points $\sigma \in \Aut_{\X}(k)$ such that $\sigma^* E \cong E$. From this, we see that $\Phi^0(\E)(k)$ is a subgroup of $\Aut^0_{\X}(k)$. By the next proposition, $\Phi^0(\E)$ is a closed subspace of $S$, so we can define the scheme structure on it as the reduced subscheme structure of $S$. We denote the neutral connected component of $\Phi^0(\E)$ by $\Phi^{00}(\E)$.
 \end{remark}

 \begin{proposition}
 \label{dimension}
     $\Phi^0(\E)$ is a closed subset of $S$. Equip it with the reduced structure, $\Phi^0(\E)$ is a closed subgroup of $\Aut^0_{\X}$.
 \end{proposition}

 \begin{proof}
     $\Phi^0(\E)$ is the image of a map from the etale cover of \underline{Isom} to $S$, which is constructible.

     We also know that $\Phi^0(\E)(k)$ is closed under multiplication, so $\Phi^0(\E)(k)$ is a subgroup of $S(k)$. Let $\overline{\Phi^0_{\X}(\E)}$ be the closure with the reduced structure. Since $k = \bar{k}$, $k$-points of the constructible set $\Phi^0(\E)$ is dense, so the closure $\overline{\Phi^0(\E) \times \Phi^0(\E)} = \overline{\Phi^0(\E)(k) \times \Phi^0(\E)(k)}$ in $S$. This implies that $\overline{\Phi^0(\E)}$ is closed under multiplication, that is, $m_{\overline{\Phi^0(\E)} \times \overline{\Phi^0(\E)}}$ factors through $\overline{\Phi^0(\E)} \hookrightarrow S$ (since $\overline{\Phi^0(\E)} \times \overline{\Phi^0(\E)}$ is reduced.) So $\overline{\Phi^0(\E)}$ is a closed subgroup scheme of $S$. We see that $\Phi^0(\E)$ contains a dense open set $U$ in $\overline{\Phi^0(\E)}$, so the translations of this open dense subset covers the $k$-points of $\overline{\Phi^0(\E)}$. 
     
     We claim that $UU(k): = m(U,U)(k) = \overline{\Phi^0(\E)}(k)$. Given any $g \in \overline{\Phi^0(\E)}(k)$, $gU^{-1}$ is open, so $gU^{-1} \cap U \neq \varnothing$. This implies that there exists $u, v \in U(k)$ such that $gu^{-1} = v$, so $g = uv$.

    Together with the fact that $\Phi^0(\E)(k)$ is closed under multiplication, this shows $\Phi^0(\E)(k) = \overline{\Phi^0(\E)}(k)$. $\Phi^0(\E)$ is constructible, so $\Phi^0(\E)$ contains a dense open subset containing all the $k$-points, and since $k = \bar{k}$, we conclude that $\Phi^0(\E) = \overline{\Phi^0(\E)}$.
 \end{proof}

 \begin{remark}
    When $\E$ is a simple vector bundle under the condition that $\Char{k} = 0$, $\Phi^{00}{(\E)}$ agrees with the definition $\Phi_{\X}^{00}(\E)$ in the last subsection.

    When $\E$ is not a simple vector bundle, it is not necessarily true that $\Phi^0_{\E}$ with the reduced subscheme (of $\Aut^0_{\X}$) structure represents the functor we discussed at the beginning of the section. The notion is still a useful notion, as they agree set-theoretically, which leads to the following proposition.
 \end{remark}

\begin{proposition}
    A $1$-twisted vector bundle $\E$ on $\X$ is semi-homogeneous if and only if $\dim(\Phi_{\X}^{0}(\E)) = d$.
\end{proposition}

\begin{proof}
    We see that $\ker(\Phi^{00}(\E) \rightarrow X)$ is contained in the group of $\rk{(\mathcal{H}^i(\E))}$-torsion line bundles on $X$, which is a finite group scheme, so $\Phi^{00}(\E) \rightarrow X$ is surjective if and only if $\dim \Phi^0(\E) =\dim \Phi^{00}(\E) = d$.

    We also see that $\E$ is semi-homoegenous if and only if $\Phi^0(\E) \rightarrow X$ is surjective on $k$-points. Since both $\Phi^0(\E)$ and $X$ are proper, surjectivity on $k$-points implies surjectivity. Since $\Phi^0(\E)$ is a group scheme, this is equivalent to the surjectivity of $\Phi^{00}(\E) \rightarrow X$.

    Therefore, $\E$ on $\X$ is semi-homogeneous if and only if $\dim(\Phi_{\X}^{0}(\E)) = d$.
\end{proof}

 % \begin{proposition}
 %     $E \in \D(X)$ is semi-homogeneous if and only if $\dim(\Phi^0{(E)}) = \dim(X)$.
 % \end{proposition}
 
 % \begin{proof}
 %     Equip $\Phi^0(E)$ with the reduced subscheme structure. Via the projection $X \times X^{\vee} \rightarrow X$, we get a map $p: \Phi^0(E) \rightarrow X$ which is a continuous group homomorphism. 
     
 %     We see that $\ker{p}(k)$ consists of points $\{\LL \in X^{\vee}| E \otimes \LL \cong E \}$, which is finite. This shows that $\dim(\ker{p}) = 0$. So $\dim(\im{p}) = \dim(\Phi^0(E))$. Hence, our proposition is proved.

 % \end{proof}

\section{Semi-homogeneous Vector Bundles}
In this section we assume $k = \bar{k}$, $\Char(k) = 0$. So by choosing a $k$-point of $X$, $X \cong A$ as abelian varieties. Many of the tools used in this section come from \cite{mukai1978semi}.

\subsection{Behavior under Isogenies}
\begin{proposition}[direct summand of semi-homogeneous is semi-homogeneous]
    If $\E \cong \E_1 \oplus \cdots \oplus \E_n$ is semi-homogeneous, then $\E_i$ is semi-homogeneous for all $i$. 
    %This is to say, direct summands of semi-homogeneous sheaf is semi-homogeneous.
\end{proposition}

\begin{proof}
    Notice first that $\Phi^{00}(\E) \cap \Phi^{00}(\E_i) \neq \emptyset$ for all $i$ since they all contain the identity element in $\Aut^0_{\X}(k)$.
    
    If $\Phi^{00}(\E) \not \subset \Phi^{00}(\E_i)$, then there exists an infinite sequence $\sigma_1, \sigma_2, \cdots$ of elements in $\Phi^{00}(\E)(k)$ such that they are in different cosets of $\Phi^{00}(\E_i)(k) \cap \Phi^{00}(\E)(k)$ in $\Phi^{00}(\E)(k)$. 
    
    We see that given any $\sigma_s$ and $\sigma_t$ in this sequence, $\sigma_s^*{\E_i} \not \cong \sigma_t^*{\E_i}$, as otherwise $(\sigma_s-\sigma_t)^*{\E_i} \cong \E_i$, which would imply that $\sigma_s$ and $\sigma_t$ are in the same coset of $\Phi^{00}(\E_i)$. This means that we have infinitely many non isomorphic direct summands of $\E$, namely $\sigma_1^*(E_i), \sigma_2^*(E_i), \cdots$, which contradicts Krull-Schmidt theorem \cite[Theorem 1]{atiyah1956krull}.

    So \[\dim{\Phi^{00}(\E_i)} \geq \dim{\Phi^{00}(\E)} \geq d,\] which concludes the result.
\end{proof}

\begin{remark}
    The Krull-Schmidt theorem applies in our setting because $\Coh(\X)^{(1)}$ satisfies the conditions in \cite[Corollary of Lemma 3]{atiyah1956krull}. And the exact proof of \cite[Lemma 9]{atiyah1956krull} shows that any direct factor of a locally free $1$-twisted sheaf on $\X$ is also locally free.
\end{remark}

Let $\pi: Y \rightarrow X$ be an isogeny, let $\Y = \X \times _X Y$, and let $\tilde{\pi} : \Y \rightarrow \X$ be the base change of $\pi$ under $p: \X \rightarrow X$. For simplicity, we sometimes abuse the notation and use $\pi$ for $\tilde{\pi}$.

\begin{proposition}
\label{pullback}
     Given 1-twisted vector bundles $\E$ on $\X$ and $\F$ on $\Y$, $\tilde{\pi}^* {\E}$ is semi-homogeneous if and only if $\E$ is semi-homogeneous, and $\tilde{\pi}_* {\F}$ on $\Y$ is semi-homogeneous if and only if $\F$ is semi-homogeneous.
\end{proposition}

\begin{proof}
    The first step is to show that the pushforward and pullback of a semi-homogeneous vector bundle is semi-homogeneous, that is, if $\E$ (resp. $\F$) is semi-homogeneous, 
    then $\tilde{\pi}^* \E$ (resp. $\tilde{\pi}_* {\F}$) is semi-homogeneous. Given any $\sigma \in \Aut^0_{\Y}(k)$ with image $t_a \in \Aut^0_Y(k)$, let $\gamma$ be a lift of $t_{\pi(a)} \in \Aut^0_{\X}(k)$. We see that $\sigma$ and $\pi^*{(\gamma)}$ are both mapped to $t_a \in \Aut^0_{\Y}(k)$, so by \cite[Theorem 1.1]{olsson2025twisted}, given any sheaf $F$, there exists some line bundle $\LL \in \Pic^0_Y(k)$ such that \[\sigma^*F \cong (\pi^*(\gamma))^* F \otimes \LL.\]\\
    So we have \[\sigma^*(\tilde{\pi}^* \E) \cong (\tilde{\pi}^*(\gamma))^*{\tilde{\pi}^*} \E \otimes \LL \cong \tilde{\pi}^* \gamma^* \E \otimes \LL \cong \tilde{\pi}^*(\E \otimes \N) \otimes \LL \cong \tilde{\pi}^*{\E} \otimes (\tilde{\pi}^* \N \otimes \LL)\] for some $\N \in \Pic^0_X$. This shows that $\tilde{\pi}^* \E$ is semi-homogeneous.

    When $\F$ is semi-homogeneous, given any $ \sigma \in \Aut^0_{\Y}(k)$ we know that there exist some $\LL \in \Pic^0_Y(k)$ such that $\sigma^*\F \cong \F \otimes \LL$, since $\pi^\vee \colon \Pic^0_X \rightarrow \Pic^0_Y$ is an isogeny, we know that there exists some 0-twisted line bundle $\M$ on $\X$ such that $\sigma^*\F \cong \F \otimes \tilde{\pi}^* \M$. \\
    Given any $\gamma \in \Aut^0_{\X}(k)$, we can find a $\sigma \in \Aut^0_{\Y}(k)$ such that \[\gamma^* \tilde{\pi}_* \F \cong \tilde{\pi}_* \sigma^* \F = \tilde{\pi}_*(\F \otimes \tilde{\pi}^* \M) \cong \tilde{\pi}_* \F \otimes \M\] This shows that $\tilde{\pi}_* \F$ is semi-homogeneous.

     Now we assume $\tilde{\pi}^* \E$ is semi-homogeneous. To show that $\E$ is semi-homogeneous, we see that $\tilde{\pi}_* \tilde{\pi}^* \E \cong \oplus_{\LL \in \ker(\pi^\vee)}\E \otimes \LL$ is semi-homogeneous, so as a summand, $\E$ is semi-homogeneous.

    Finally we show that if $\tilde{\pi}_*\F$ is semi-homogeneous, then $\F$ is semi-homogeneous. Notice that $\Y \times_{\X} \Y \cong \Y \times \ker(\pi)$ where the two projection maps are projection, $p$ to $\Y$, and the twisted action of $\ker{\pi}$ on $\Y$, $m|_{\Y \times \ker(\pi)}$ to $\Y$ (see the remark below.) So we have \[\tilde{\pi}^* \tilde{\pi}_* \F \cong p_* m^* \F \cong \oplus_{a \in \ker{\pi}(k)} a \cdot \F\] where $a \cdot (-)$ is the descent action, see the remark below.
    The second isomorphism is by the fact that $\Char{k} = 0$ hence $\ker{\pi}$ is etale. 
    By what we showed above and the assumption, we see that $\tilde{\pi}^* \tilde{\pi}_* \F$ is semi-homogeneous, so as a direct summand (when $a = 0$, the action is trivial), $\F$ is semi-homogeneous as desired.
\end{proof}

\begin{remark}
    In the case for $\tilde{\pi}: \Y \rightarrow \X$ where $\Y \cong \BGm$, the descent action of $\ker{\pi}$ on $\tilde{\pi}^* \F$ is twisted by a line bundle $\delta_{\ker{\pi}}$. More concretely, when $\ker{\pi}$ is etale, for any $\F$ on $\X$, $a \in \ker{\pi}(k)$, $a \cdot \tilde{\pi}^*\F \cong t_a^* \tilde{\pi}^*\F \otimes \delta_a$. This is by \cite[8.11]{olsson2025twisted}.
\end{remark}

\begin{theorem}
\label{linebundle}
If $\E$ is a $1$-twisted simple semi-homogeneous vector bundle (recall that simple means $\Hom(\E,\E) = k$,) there exists an isogeny $f : X' \rightarrow X$ such that $\X' := \X \times_X X' \cong \BGm_{,X'}$ and $\E \cong f_* \LL$ for some 1-twisted line bundle on $\X'$ (here we also denote the base change of $f$ under $\pi : \X \rightarrow X$ by $f$.)
\end{theorem}

\begin{proof}
    Let $\Phi_0(\E) = \ker(h: \Phi^{00}(\E) \rightarrow X)$. $\Phi_0(\E)$ is a subgroup scheme of $X^\vee$ and $\Phi_0(\E)(k)$ is the group of line bundles $\LL$ such that $\E \otimes \LL \cong \E$. Let $G$ be a simple subgroup of $\Phi_0(\E)$, let $Y : = (X^\vee / G)$, and let $X': = Y^\vee$. Since we have the isogeny $X^{\vee} \rightarrow Y$, we get the map $X' \rightarrow X$ by dualizing. 
    
    We know that the order of $G$, $l$, is prime.
    
    We have \[f_* f^* \E \cong f_*(\mathcal{O}_{\X'} \otimes f^* \E) \cong \E \otimes f_* f^* \mathcal{O}_{\X'} \cong \oplus_{\LL \in \ker{f^\vee}(k)} \E \otimes \LL \cong \E^{\oplus l}\]
    This implies that $\Hom_{\X'}(f^* \E, f^* \E) \cong \Hom_{\X}(\E, f_*f^*\E) \cong \oplus_{\LL \in \ker{f^\vee}(k)}\Hom_{\X}(\E, \E \otimes \LL)$ which is isomorphic to $k[x]/(x^l -1)$ as $k$-algebras (since $l$ is prime, $\ker{f^\vee}(k) \cong \mathbb{Z}/l \mathbb{Z}$). So we have \[f^* \E \cong \E_1 \oplus \cdots \oplus \E_l\] for simple non-isomorphic $E_i$'s. This shows that we also have \[f_* f^* \E \cong f_* \E_1 \oplus \cdots \oplus f_* \E_l.\]
    By Krull-Schmitz theorem, we have $f_* \E_i \cong \E$ for all $i$.

    \begin{lemma}
        $\ord(\Phi_{0,\X'}(\E_i)) \leq \frac{\ord(\Phi_{0,\X}(\E))}{\ord{(G)}}$.
    \end{lemma}
    \begin{proof}
        Note that all groups schemes in the statement are discrete, so it suffices to show the inequality for the $k$-points.
        
        Given any $\LL \in \Phi_{0,\X'}(\E_i)$, $\LL 
        \cong f^* {\N}$ for some $\N \in \Pic^0{X}$. We have that
        \[f_*(\E_i \otimes \LL) \cong f_* (\E_i \otimes f^* \N) \cong f_* \E_i \otimes \N \cong \E \otimes \N \cong \E \cong f_* \E_i\]

        Here we think of $\LL, \N$ as 0-twisted line bundles on $\X$ and $\X'$ resp. This implies that $\N \otimes \E \cong \E$, and we know that there are precisely $\ord{(G)}$ many such $\N$ in $\Pic^0(X)$ in the preimage of $\{\LL\}$.
    \end{proof}

    Repeating the process of quotienting out by simple subgroups of $\Phi_0(\E_i)$, using the lemma one can see that after finitely many times we get an isogeny such that $\Phi_0(\E_i) = 0$.

    Then we conclude the theorem by the following lemma:

    \begin{lemma}
        If $\Phi_0(\E) = 0$ and $\E$ is simple semi-homogeneous, then $\E$ is a line bundle.
    \end{lemma}
    \begin{proof}
        We see that $\Phi_0(\E) = \ker(h: \Phi^{00}(\E) \rightarrow X)$, so the assumption that $\Phi_0(\E) = 0$ implies that $h$ is an isomorphism. By discussions in Section 2 we see that the map $\tilde{h}: \widetilde{\Phi^{00}(E)} \to \AAut^0_{\X} \xrightarrow[]{\sigma_0} \X$ is the base change of $h$ via $\X \to X$, hence is also an isomorphism.

        Now take the universal object $\gamma$ on $\widetilde{\Phi(\E)} \times \X$, which is an automorphism $\widetilde{\Phi(\E)} \times \X \to \widetilde{\Phi^{00}(\E)} \times \X$ such that $\gamma^* \E_{\widetilde{\Phi^{00}(\E)}} \cong \E_{\widetilde{\Phi^{00}(\E)}} \otimes M$ for the weight 1 line bundle $M := \HHom_{\widetilde{\Phi^{00}(\E)} \times \X}(\E, \gamma^* \E)$ on $\widetilde{\Phi^{00}(\E)}$.

        As in Remark \ref{restriction to 0}, there is a lift of $0: \Spec k \to X$ to $\Spec k \to \X$. Pullback via this map, we get $\tilde{h}^* \E \cong \mathcal{O}_{\widetilde{\Phi^{00}(\E)}}^{\oplus \rk{\E}} \otimes M$.

        Since $\tilde{h}$ is an isomorphism and $\E$ is simple, we see that $\rk{\E}$ must be 1.
    \end{proof}
\end{proof}

\begin{remark}
    Note that the existence of $M$ shows that $\widetilde{\Phi^{00}(\E)} \to \Phi^{00}(\E)$ is always the trivial gerbe.
    
    If we remove the assumption that $\Phi_0(\E) = 0$, the same proof shows that $\tilde{h}^* \E \cong M^{\oplus \rk{\E}}$.
\end{remark}

\begin{corollary}
\label{distinct}
    Let $\E$ be a $1$-twisted simple semi-homogeneous vector bundle, $f$ the isogeny in the theorem and $\LL$ be the line bundle such that $\E \cong f_* \LL$. Then $t_a^* \LL \otimes \delta_a \not \cong \LL$ for all nontrivial $a \in \ker(f)(k)$. 
\end{corollary}

\begin{proof}
     We have that \[f^* \E \cong f^*f_* \LL \cong \bigoplus_{a \in \ker(f)} (t_a^*\LL \otimes \delta_a).\]
    As showed in Theorem \ref{linebundle}, we also have \[f^* \E \cong \LL_1 \oplus \LL_2 \cdots \LL_{l}\] where all summands are simple and distinct.
    Then by the Krull-Schimidt theorem for locally free sheaves, we know that the first decomposition is a direct sum of $l$ distinct simple sheaves, hence $t_a^* \LL \otimes \delta_a \not \cong \LL$ for all $0 \neq a \in \ker(f)(k)$.
\end{proof}
\subsection{Decompositions}

Let $\pi: Y \rightarrow X$ be an isogeny such that $\Y_n : = \X_n \times_{X, \pi} Y$ is the trivial $\mu_n$-gerbe over $Y$, and hence $\Y := \X \times_X Y$ is the trivial $\mathbb{G}_m$-gerbe. For simplicity, we abuse the notation and also denote $\pi: \Y \rightarrow \X$.

\[\begin{tikzcd}
	{\Y \times_k \ker(\pi)} & {\Y \times_{\X} \Y} & \Y \\
	& \Y & \X
	\arrow["\cong", from=1-1, to=1-2]
	\arrow["m", curve={height=-18pt}, from=1-1, to=1-3]
	\arrow["p", from=1-1, to=2-2]
	\arrow[from=1-2, to=1-3]
	\arrow[from=1-2, to=2-2]
	\arrow[from=1-3, to=2-3]
	\arrow[from=2-2, to=2-3]
\end{tikzcd}\]

Let $\mathcal{C}$ be category that consists of objects and morphisms on $\Y$ with descent data to $\X$. Concretely, objects in $\mathcal{C}$ are pairs $(\F, \gamma)$ where $\F$ is a $1$-twisted coherent sheaf on $\Y$ and $\gamma: m^* \F \cong p^*\F$ a descent data that satisfies the cocycle conditions (cf. \cite[4.2.1.1]{olsson2023algebraic}); morphisms in $\mathcal{C}$ are the sheaf morphisms that respect the descent data, that is, $(\F, \gamma) \rightarrow (\F', \gamma')$ is a morphism $f: \F \rightarrow \F'$ on $\Y$ such that the following diagram commutes.
    \[\begin{tikzcd}
	{m^* \F} & {p^*\F} \\
	{m^* \F'} & {p^*\F'}
	\arrow["\gamma", from=1-1, to=1-2]
	\arrow["{m^*(f)}"', from=1-1, to=2-1]
	\arrow["{p^*(f)}", from=1-2, to=2-2]
	\arrow["{\gamma'}", from=2-1, to=2-2]
    \end{tikzcd}\]

\begin{remark}
    In the case where $\pi$ is separable (eg. $\Char{k} = 0$,) since $\ker{\pi}$ is a finite group scheme, it is discrete. So we can equivalently define the objects in $\mathcal{C}$ as $(\F, \{\gamma_a\}_{a \in \ker{\pi}(k)})$, where $\gamma_a : t_a^* \F \otimes \delta_a \rightarrow \F$ on $\Y$ and satisfies $\gamma_a \circ (a \cdot \gamma_b) = \gamma_{ab}$. And correspondingly, morphisms $(\F, \{\gamma_a\}_{a \in \ker{\pi}(k)}) \rightarrow (\F', \{\gamma'_a\}_{a \in \ker{\pi}(k)})$ are $f: \F \rightarrow \F'$ such that the following diagram commutes. 
    \[\begin{tikzcd}
	{t_a^* \F \otimes \delta_a} & \F \\
	{t_a^* \F' \otimes \delta_a} & {\F'}
	\arrow["{\gamma_a}", from=1-1, to=1-2]
	\arrow["{t_a^*(f) \otimes \delta_a}"', from=1-1, to=2-1]
	\arrow["f", from=1-2, to=2-2]
	\arrow["{\gamma_a'}", from=2-1, to=2-2]
\end{tikzcd}\]

\end{remark}

\begin{definition}
    Fix an ample line bundle of $Y$, we say $\F \in \Coh(\Y)^{(1)}$ is semi-stable if $\F \otimes \chi^{-1} \in \Coh(Y)$ is semi-stable (c.f \cite[Definition 1.2.4]{huybrechts2010geometry}). Here $\chi$ is the universal 1-twisted line bundle on $\Y_n$. 

    We define the reduced Hilbert polynomial of $\F$ to be $p(\F) : =p(\F \otimes \chi^{-1})$. 
\end{definition}

\begin{proposition}
\label{filtration}
     Given any $\E \in \mathcal{C}$ semi-stable, there is a filtration \[0 = \F_0 \subset \F_1 \subset \cdots \subset \F_n = \E\] such that 
    \begin{enumerate}
        \item All of $\F_i$ and $\E_i := \F_i/\F_{i-1}$ with compatible descent data with $\E$ are in $\mathcal{C}$;
        \item $\F_i$'s and $\E_i$'s are semi-stable with $p(\F_i) = p(\E_i) = p(\E)$ where $p(-)$ denotes the reduced Hilbert polynomial on $Y$ (c.f \cite[Definition 1.2.3]{huybrechts2010geometry});
        \item $\E_i$'s are simple in $\mathcal{C}$ (ie. $\Hom_{\mathcal{C}}(\E_i, \E_i) = k$.)
    \end{enumerate}
\end{proposition}

\begin{proof}
    If $\Hom_{\mathcal{C}}(\E, \E)= k$, we are done. 
    
    If not, there exists a $f: \E \rightarrow \E$ in $\mathcal{C}$ that is neither surjective (injective) or $0$. Indeed, if $f$ is surjective (injective) or $0$, we would have that $f$ is an isomorphism or the zero map. Since $k = \bar{k}$, and $\End_{\mathcal{C}}(\E)$ is finite dimensional over $k$, this would imply that $\Hom_{\mathcal{C}}(\E, \E)= k$.

    Let $\F = \im{f}$, we have $\F \in \mathcal{C}$. By \cite[Proposition 1.2.6]{huybrechts2010geometry}, $p(\F) \geq p(\E)$. By \cite[Proposition 6.13]{mukai1978semi}, $\E$ is semi-stable, we then have $p(\F) = p(\E)$. $\E$ is semi-stable, hence pure, we see that since $\ker{f} \not \cong 0$, $\rk({\ker{f}}) \neq 0$, so $\rk{\F} < \rk{\E}$.

    Let $\tilde{\F}$ be the saturation of $\F \subset \E$. Since $\tilde{\F}/\F$ is torsion, we have $p(\E) = p(\F) \leq p(\tilde{\F})$. Since $\E$ is semi-stable, we have $p(\tilde{\F}) = p(\F)$, which implies that $\tilde{\F}/\F = 0$, hence $\tilde{\F} = \F$. Therefore, $Q = \E/\F$ is torsion free and $p(\E/\F) = p(\F) = p(\E)$.

    Now we show that $Q = \E/\F$ is semi-stable. We already showed that it is torsion free, so it suffices to show that for any nontrivial subsheaf $Q' \subset Q$ with $\rk{Q'} < \rk{Q}$, where $p(Q') \leq p(Q)$. 
    For any such subsheaf $Q' \subset Q$ with $\rk{Q'} < \rk{Q}$, there is an exact sequence \[0 \rightarrow \F \rightarrow \tilde{Q}' \rightarrow Q' \rightarrow 0\] for some $\tilde{Q}' \subset \E$. Since $Q$ is torsion free, $\rk{Q'} > 0$. So $p(\F) = p(\E) \geq p(\tilde{Q}')$, we then have $p(\F) \geq p(\tilde{Q}') \geq p(Q')$. So $p(Q) = p(\F) \geq p(Q')$ as desired.

    We now have $0 \subset \F \subset \E$ such that $\F \xhookrightarrow{} \E$, and hence $\E/\F$, are in $\mathcal{C}$ and semi-stable with $p(\E/\F) = p(\F) = p(\E)$. Repeating this process, we get the desired filtration. $\Hom_{\mathcal{C}}(\E_i, \E_i)= k$ since otherwise we can apply the above process to $\E_i$ to get a refined filtration that satisfies the desired conditions. 
\end{proof}

\begin{lemma}
\label{locally free filtration}
    $\F_i$ and $\E_i$'s obtained in Proposition \ref{filtration} are locally free.
\end{lemma}

\begin{proof}
    $\pi^* \E$ is semi-homogeneous, so $\pi^* \E \otimes \chi^{-1}$ is also semi-homogeneous. By \cite[6.13]{mukai1978semi}, $\pi^* \E$ is semi-stable, and we have a canonical descent data for $\pi^* \E$.
    
    By Proposition \ref{filtration}, we have a filtration \[0 = \F_0 \subset \F_1 \subset \cdots \subset \F_n = \pi^*(\E)\] that can be refined into the Jordan-Holder filtration of the semi-stable vector bundle $\pi^*\E \otimes \chi^{-1}$ (cf. \cite[Definition 1.5.1]{huybrechts2010geometry})

    \begin{align*}0 = \F_0 &= F_{1_0} \subset F_{1_1} \subset F_{1_2} \subset \cdots \subset F_{1_{m_1}} = \F_1 \otimes \chi^{-1}  \\&= F_{2_0} \subset F_{2_1} \subset \cdots \subset F_{2_{m_2}} = \F_2 \otimes \chi^{-1} \subset \cdots \subset F_{n_{m_n}} = \F_n \otimes \chi^{-1} = \pi^* \E \otimes \chi^{-1}.\end{align*}
    such that each $E_{i_j} := F_{i_j}/F_{i_{i-1}}$ is stable and $p(E_{i_j}) = p(\pi^*{\E})$.

    As in \cite[Proposition 6.15]{mukai1978semi} through \cite[Proposition 6.19]{mukai1978semi}. The Jordan-Holder filtration of $\pi^* \E \otimes \chi^{-1}$ gives the decomposition $\pi^* \E \otimes \chi^{-1} \cong \oplus_{E_{i_j}} U_{Y, E_{i_j}}$ where $E_{i_j}$'s are simple semi-homogeneous with $\delta_{Y}(E_{i_j}) = \delta_{Y}(\pi^* \E \otimes \chi^{-1})$.

    Thus, we see that $\F_s \otimes \chi^{-1} \cong \oplus_{i \leq s} U_{Y, E_{i_j}}$, which is semi-homogeneous by \cite[Proposition 6.19]{mukai1978semi}. So $\F_s$ is semi-homogeneous on $\Y$ for all $s$.
\end{proof}

\begin{theorem}
\label{summand}

If a 1-twisted vector bundle $\E$ on $\X$ is semi-homogeneous then $\E \cong \bigoplus_{Q_{\sigma}}U_{\X, Q_{\sigma}}$ where $U_{\X, Q_{\sigma}}$ is a filtration whose successive quotients are $Q_{\sigma}$, and $Q_\sigma$'s are non-isomorphic simple semi-homogeneous, all $\delta(Q_{\sigma}):= \frac{\det(Q_{\sigma})}{\rk(Q_{\sigma})}$ are equal in $NS(X) \otimes_{\mathbb{Z}} \mathbb{Q}$, and only finitely many $U_{\X, Q_{\sigma}} \neq 0$. 
\end{theorem}

\begin{proof}

    We fix an isogeny $\pi \colon Y \rightarrow X$ that trivializes the gerbe $\X_n$. We see that $\delta_{\X}(\E_i):= \frac{\det(Q_i)}{rk(\E_i)}$ are equal in $NS(X) \otimes_{\mathbb{Z}} \mathbb{Q}$, and only finitely many $U_{\X, Q_i} \neq 0$ implies $\E$ being semi-homogeneous. This follows from the fact that $\delta_{\Y}(\pi^*Q_i) = \pi^* \delta_{\X}(Q_i))$, \cite[Theorem 6.19]{mukai1978semi} and Proposition \ref{pullback}.

    Now we show the forward direction.

    By Proposition \ref{filtration} and Lemma \ref{locally free filtration} \[0 = \F_0 \subset \F_1 \subset \cdots \subset \F_n = \pi^*(\E)\] is a filtration of vector bundles in $\mathcal{C}$, it descents to a filtration of $1$-twsited vector bundles \[0 = \G_0 \subset \G_1 \subset \cdots \subset \G_n = \E\] on $\X'$ such that $\pi^*\G_i = \F_i$. We see that $Q_i := \G_i/\G_{i-1}$ satisfies $\pi^*Q_i \cong \E_i$ and $\delta_{\X}(Q_i) = \delta_{\X}(\E)$. By Proposition \ref{pullback}, $\G_i$ and $Q_i$ are semi-homogeneous for all $i$. 

    We conclude that $\E \cong \bigoplus_{Q_{\sigma} \in \{Q_i\}_i/{\cong}}U_{\X, Q_{\sigma}}$ by Corollary \ref{relation}. 
\end{proof}

\begin{remark}
    We will not use Theorem \ref{summand} until we show Corollary \ref{relation}. In fact, it is not used anywhere in this paper.
\end{remark}

\section{Semi-homogeneous Complexes}
    Having studied the properties of semi-homogeneous vector bundles in $\X$, we can now proceed to analyze the behavior of semi-homogeneous complexes using the structural results proved in Section 3. Many of the tools in this section are inspired by \cite{de2022point}. In this section, we assume $\Char(k) = 0$ and $k = \bar{k}$.

    \begin{definition}
     If $k$ is algebraically closed, an object $E \in \D(\X)^{(1)}$ on $\X$ is called $\textbf{semi-homogeneous}$ if for every $\sigma \in \Aut^0_{\X}(k)$ there exists a 0-twisted line bundle $\LL$ such that \[\sigma^* E \cong E \otimes \LL.\] 
    For general $k$, we call a $E \in \D(\X)^{(1)}$ semi-homogeneous if its base change to an algebraic closure $\bar{k}$ of $k$ is semi-homogeneous.
 \end{definition}

 \begin{remark}
     When $\X \cong \BGm_{,X}$, this defines semi-homogeneous complexes on $X$(c.f. \cite[Definition 3.2]{de2022point}.)
 \end{remark}

 Given any fixed complex $E$ such that $\Ext^i(E, E) = 0$ for all $i < 0$, there is a map $\AAut^0_{\X} \to \D_{\X}$ by sending $\sigma$ to $\sigma^* E$. We call the induced map on the coarse spaces $\gamma_E$.
 
\begin{definition}
    Given any fixed complex $E$ such that $\Ext^i(E, E) = 0$ for all $i < 0$, we define $\Phi(E)$ to be the fiber of the map $\gamma_E$ at the $k$-point $[E] \in D_{\X}$. $\Phi^{00}(E)$ to be the neutral connected component of $\Phi(E)$.
\end{definition}
 We see that they are both subgroup schemes of $\Aut^0_{\X}$.

 Take some $\mathcal{H}^i(E) \neq 0$. Since semi-homogeneity passes to cohomology sheaves, if $E$ is semi-homogeneous, $\mathcal{H}^i(E)$ is also semi-homogeneous. (Semi-homogeneity also implies that the locally free locus of $\mathcal{H}^i(\FF)$ is $\X$.) 

 \begin{lemma}
 \label{stab dimension for complex}
     Assume $E$ is semi-homogeneous such that $\Ext^i(E, E) = 0$ for all $i < 0$ and $\mathcal{H}^i(E) \neq 0$, we have $\Phi(E) = \Phi(\mathcal{H}^i(E))$.
 \end{lemma}

 \begin{proof}
    $\Phi(E)(k) \subset \Phi(\mathcal{H}^i(E))(k)$ since any $\sigma$ that stabilizes $E$ must stabilize its cohomology sheaves.
    Since they both are subgroup schemes of $\Aut^0_{\X}$ over charateristic 0, they are both closed subschemes with reduced structure. So $\Phi(E) \subset \Phi(\mathcal{H}^i(E))$ is also a closed subscheme.
    Since $E$ is semi-homogeneous, we have that the map $\Phi(E) \to X$ is surjective, which implies that $\dim \Phi(E) = d = \dim \Phi(\mathcal{H}^i(E))$. We then conclude the lemma.   
 \end{proof}

\subsection{Analyzing Semi-homogeneous Complexes on a Trivialization of $\X$}

Let $\FF$ be a 1-twisted semi-homogeneous complex on $\X$ such that $\Ext^i(\FF, \FF) = 0$ for all $i < 0$. 

Take some $\mathcal{H}^i(\FF) \neq 0$, by Lemma \ref{locally free filtration}, there is a simple semi-homogeneous vector bundle $\E$ such that $\delta_{\X}(\E) = \delta_{\X}(\mathcal{H}^i(\FF))$.

\begin{proposition}
\label{homogeneous}
    Let $\FF$ be a $1$-twisted semi-homogeneous complex on $\X$ $\Ext^i(\FF, \FF) = 0$ for all $i < 0$, with $\mathcal{H}^{i}(\FF) \neq 0$ for some $i$. If there exists a $1$-twisted simple semi-homogeneous vector bundle $\E$ such that $\delta_{\X}(\E) = \delta_{\X}(\mathcal{H}^i(\FF))$, then there is an isogeny $f : X' \rightarrow X$ and a line bundle $\N$ on $\X' : = \X \times_{X, f} X'$ such that $\X' \rightarrow X'$ is the trivial $\Gm$-gerbe and $f^*\FF \otimes \N^{-1}$ is homogeneous on $\X'$.

    Moreover, $f^* \FF \cong \bigoplus_{\LL \in X'^{\vee}}U_{\LL} \otimes \LL$ where $U_{\LL}$ is a homogeneous complex admitting a filtration whose successive quotients are of the form $\mathcal{O}_{X'}[s]$ for various integers $s$.
\end{proposition}

\begin{remark}
\label{trivial}
    Since $\X'$ is trivial, tensoring by $\chi^{-1}$ (tensoring down by a weight $-1$ line bundle) gives an equivalence between the 1-twisted coherent sheaves and the 0-twisted ones. So, from now on, we view $1$-twisted sheaves on $\X'$ as sheaves on $X'$ (that is, for simplicity, we drop $- \otimes \chi^{-1}$.) This also means that the results in \cite{mukai1978semi} and \cite{de2022point} apply.
\end{remark}

\begin{proof}
    Let $f: \X' \rightarrow \X$ be the base change of the isogeny we get from Theorem \ref{pushforward} such that $\E \cong f_* \LL$ for some line bundle $\N$ on $\X'$.

    Similarly to Proposition \ref{pullback}, one can show that $f^* \FF$ is semi-homogeneous, which implies that \[\Phi^{00}_{\X'}(f^*{\FF}) = \Phi^{00}_{\X'}(f^* {(\mathcal{H}^s{(\FF)})})\] by Lemma \ref{stab dimension for complex}.

    By \cite[Lemma 6.8]{mukai1978semi}, we have $\Phi^{00}_{\X'}(f^*\mathcal{H}^i(\FF)) = \Phi^{00}_{\X'}(f^*\E)$, which implies that $\Phi^{00}_{\X'}(f^*\FF) = \Phi^{00}_{\X'}(f^*\E)$.

    Let $K = f^* \FF$. By Theorem \ref{linebundle}, we know $f^*\E \cong \N \oplus \N_1 \oplus \cdots \oplus \N_{l-1}$ where $l$ is the order of $f$ and all summands are distinct line bundles. By \cite[Lemma 3.11]{mukai1978semi}, $\Phi_{\X'}^{00}(f^*{\E}) \subset \Phi^{00}_{\X'}({\N})$.

    We then conclude by the following lemma.

\begin{lemma}
    $K\otimes \N^{-1}$ is homogeneous, hence $K \otimes \N^{-1} \cong \bigoplus_{\LL \in X'^t}U_{\LL} \otimes \LL$ where $U_{\LL}$ is a homogeneous complex admitting a filtration whose successive quotients are of the form $\mathcal{O}_{X'}[s]$ for various integers $s$.
\end{lemma}
\begin{proof}
     Let $\lambda_{\N} : X' \rightarrow X'^{\vee}$ be $\lambda_{\N}(a)= t_a^* \N \otimes \N^{-1}$. By the discussion above, we know that $t_a^*K  \cong K \otimes \lambda_{\N}(a)$ for every $a \in X'(k)$.
    Given any $a \in X'$, we have \[t_a^*(K \otimes \N^{-1}) \cong K \otimes \lambda_{\N}(a) \otimes t_a^* \N^{-1} \cong K \otimes t_a^* \N \otimes \N^{-1} \otimes t_a^* \N^{-1} \cong K \otimes \N^{-1}\]
    Hence $K \otimes \N^{-1}$ is homogeneous. By \cite[2.2]{de2022point}, $K \otimes \N^{-1} \cong \bigoplus_{\LL \in X'^{\vee}}U_{\LL} \otimes \LL$.
\end{proof}
\end{proof}

% \textbf{For the rest of the section, we assume that $\ker{f}$ is etale.}

For any $\GG' \in \D(\X')^{(1)}$, $a \in \ker{f}(k)$, write $a \cdot \GG' := t_a^* \GG' \otimes \delta_a$ where $\delta_a$ as in \cite[8.11]{olsson2025twisted}. And by \cite[8.11]{olsson2025twisted}, when $\GG' \cong f^* \GG$ for some $\GG \in \D(\X)^{(1)}$, this is the descent action.

For the $f$ and $\E$ as in the previous proposition, by Corollary \ref{distinct} this action induces a free action of $\ker{f}(k)$ on $\{\N, \N_1, \cdots, \N_{l-1} \}$ which sends $\N$ to some $\N_i$, that is, we have 

\begin{proposition}
\label{inject}
$a \cdot \N \not \cong \N$ for all nontrivial $a \in \ker(f)(k)$.
\end{proposition}

Let $G := X'^\vee(k)/\{a \cdot \N \otimes \N^{-1}\}_{a \in \ker{(f)}(k)}$, let $\tilde{\Sigma}(\FF) \subset X'^{\vee}$ be the line bundles $\LL$ such that $U_{\LL} \neq 0$ ($U_{\LL}$ as in the decomposition in Proposition \ref{homogeneous},) then let $\Sigma(\FF)$ be the image of $\tilde{\Sigma}(\FF)$ in $G(k)$. 

\begin{lemma} 
$\tilde{\Sigma}(\FF)$ is the preimage of $\Sigma(\FF)$ under $X'(k) \rightarrow G(k)$.
\end{lemma}
\begin{proof}
    By descent, $a \cdot f^*\FF \cong f^*\FF$, so \[a \cdot (\bigoplus_{\LL \in \tilde{\Sigma}(\FF)}U_{\LL} \otimes \LL \otimes \N) \cong \bigoplus_{\LL \in \tilde{\Sigma{\FF}}} U_{\LL} \otimes \LL \otimes \N\]
    so \[a \cdot (U_{\LL} \otimes \LL \otimes \N) \xhookrightarrow{} a \cdot (\bigoplus_{\LL \in X'^{\vee}}U_{\LL} \otimes \LL \otimes \N) \cong \bigoplus_{\LL \in X'^{\vee}}U_{\LL} \otimes \LL \otimes \N\]
    and we know \[a \cdot (U_{\LL} \otimes \LL \otimes \N) \cong t_a^*(U_{\LL} \otimes \LL \otimes \N) \otimes \delta_a \cong U_{\LL} \otimes \LL \otimes (t_a^* \N \otimes \N^{-1} \otimes \delta_a) \otimes \N \cong U_{\LL} \otimes \LL \otimes (a \cdot \N \otimes \N^{-1}) \otimes \N.\]
    So if $U_{\LL} \neq 0$ then $U_{\LL \otimes (a \cdot \N \otimes \N^{-1})} \neq 0$. Since $\Hom(L, L') = 0$ for $L \neq L' \in X'^{\vee}(k)$ on abelian varieties, we see that $\LL \otimes (a \cdot \N \otimes \N^{-1})$ must also be an element of $\tilde{\Sigma}(\FF)$, which means $\tilde{\Sigma}(\FF)$ is stable under the action of the subgroup we quotient out. The statement then follows.
\end{proof}

\begin{proposition}
\label{pushforward}
For the same $\FF$ and $f$ as in Proposition \ref{homogeneous}, $\FF \cong f_*(H \otimes \N)$ for some 0-twisted homogeneous complex $H$ on $\X'$.
\end{proposition}

\begin{proof}
    Let $H := \oplus_{\sigma \in \Sigma} (U_{\LL_{\sigma}} \otimes \LL_{\sigma})$, we see that we have $H \otimes \N \xhookrightarrow{} f^* \FF$. So we have nonzero map $f_* (H \otimes \N \xhookrightarrow{} \FF)$, which is an isomorphism since it is an isomorphism after applying $f^*$. 
    
    Indeed, $f^*f_* H \otimes \N \xhookrightarrow{} f^*\FF$ is the composition $f^* f_* (H \otimes \N) \to H \otimes N \to f_* \FF$ where the second map is the injection we started with. Note that $f^*f_* H \otimes \N \xhookrightarrow{} f^*\FF$ is $\ker{f}$-invariant and the restriction $H \otimes N \to f^*f_* H \oplus_{\sigma \in \tilde{\Sigma(\FF)}} U_{\LL_{\tilde{\sigma}} \otimes \LL_{\tilde{\sigma}}} \otimes \N \xhookrightarrow{} f^*\FF$ is injective. Since $\tilde{\Sigma}(F)$ is the orbit of representatives for $\Sigma(F)$, we see that the map $f^*f_* H \otimes \N \xhookrightarrow{} f^*\FF$ is injective for any summand of $f^* f_* (H) \otimes N$, hence an isomorphism by \ref{Ext}.
\end{proof}

\begin{corollary}
\label{compare}
    For some simple $1$-twisted semi-homogeneous vector bundles $\E' \neq \E$ on $\X$ with $\delta_{\X}(\E) = \delta_{\X}(\E')$, $\Sigma(\E) \cap \Sigma(\E') = \varnothing$.
\end{corollary}

\begin{proof}
    Since $\delta_{\X}(\E) = \delta_{\X}(\E')$, applying Proposition \ref{homogeneous} to $\FF = \E'$ and $\E = \E$, we see that $f^*\E' \otimes \N^{-1}$ is homogeneous, and $\E' \cong f_*(\oplus_{\sigma \in \Sigma{(\E')}}U_{\LL_{\sigma}} \otimes \N)$, $\E \cong f_* \N$. Note that since $\E'$ is a sheaf, $U_{\LL_{\sigma}}$ are also sheaves (i.e. concentrated in degree 0.)
    
    Suppose $\Sigma(\E) \cap \Sigma(\E') \neq \varnothing$. Since $\Sigma(\E)$ only contains one element, which is the class of $\N$, by Proposition \ref{pushforward}, we see that $\E'$ has a direct summand $U$ that is a filtration whose successive quotients are $\E$. Since $\E' \neq \E$, we see that either $U$ is a proper summand of $\E'$, or $\E \subset U$ is a proper subsheaf (ie. $U$ is a nontrivial extension of $\E$, this gives a nonscaler map $U \to U$). This contradicts the fact that $\E'$ is simple.
    % we then can write $\E' \cong f_*(U_{\N} \oplus (\oplus_{\sigma \in (\Sigma(\E')-{[\N]})}U_{\LL_{\sigma}}) \otimes \N)$
\end{proof}
\subsection{Properties of Semi-homogeneous Complexes}
\begin{definition}
    For $\FF$ satisfying Proposition \ref{homogeneous}, we say $\FF$ is \textbf{homogenized by $(f, \N)$}.
    % For each $\sigma \in \Aut^0_{\X}(k)$, let $\Phi_\sigma(\FF) \in \Pic_X(k)$ be the set of line bundles $\LL$ such that $\sigma^* \FF \cong \FF \otimes \LL$.
\end{definition}

\begin{lemma}
\label{Ext}
Let $\FF, \FF' \in D(\X)$ be semi-homogeneous complexes that are homogenized by the same $(f, \N)$,
\begin{enumerate}
    \item If $\Sigma(\FF) \cap \Sigma({\FF'}) = \emptyset$, then $\Ext^s_{\X}(\FF, \FF') = 0$ for all s,
    \item Assume $\FF$ and $\FF'$ are both vector bundles concentrated in degree 0. If $\Sigma(\FF) \cap \Sigma(\FF') \neq \emptyset$, then $\Hom_{\X}(\FF, \FF') \neq 0$.
\end{enumerate}
\end{lemma}

\begin{proof}
    Fix representatives of $\Sigma{(\FF)}$ and $\Sigma{(\FF')}$ in $X'(k)$. By Proposition \ref{pushforward}, we see that \[\FF \cong f_*(H_{\FF} \otimes \N)\] \[\FF' \cong f_*(H_{\FF'} \otimes \N) \] where $H_{\FF}$ and $H_{\FF'}$ are homogeneous.

    By \cite[2.2]{de2022point}, we can write $H_{\FF} \cong \oplus_{\LL} V_{\LL}$ where $V_{\LL}$ is a filtration whose successive quotients are $\LL[s]$ for various $s$.

    Similarly as in \cite[Lemma 4.3]{de2022point}, because of the canonical filtration of $\FF$ and $\FF'$, to show (1), it suffices to show for $\LL \not \cong \LL'$ we have $\Ext^s(\LL, \LL') = 0$. This is true since similarly to Remark \ref{trivial}  we can view $\LL$ and $\LL'$ as 0-twisted line bundles on the trivial gerbe $\X'$.

    (2) follows from the same ar gument and that $\Hom_{\X'}(\LL, \LL) \neq 0$ on $\X'$.
\end{proof}
\begin{remark}
    The assumption that $\FF$ and $\FF'$ are homogenized by the same $(f, \N)$ ensures that we can define $\Sigma(\FF)$ and $\Sigma(\FF')$ by the same isogeny $f: X' \rightarrow X$ and that they live in the same group $G$.
    %Need to talk about when $\Sigma{(F)}$ and $\Sigma{(F')}$ can be compared (ie. can be analyzed using the same isogeny and line bundle), and how it works for $\Sigma{(L \tilde{i}^* \tilde{i}_* F)}$ (ie. why and how its defined and its the same as $\Sigma{(F)}$)
\end{remark}

\begin{example}
\label{example}
    By our discussion earlier in the section, here are some complexes that are homogenized by the same data $(f, \N)$:
    \begin{enumerate}
        \item the truncations of some $\FF \in \D(\X)^{(1)}$, in particular, $\FF$ and $\mathcal{H}^{s}(\FF)$; 
        \item $\E$ and $\E'$ $1$-twisted vector bundles such that $\delta(\E) = \delta(\E')$.
        \item $\E$ and $\E'$ 1-twisted vector bundles such that $\Phi^{00}(\E) = \Phi^{00}(\E')$. 
        
        This can be seen as follows: Let $n$ be the order of $\alpha \in \Br(X)$, then we have a natural map $\AAut^0_{\X} \to \AAut^0_{\X^{(n)}} $ (here $\X^{(n)}$ denotes the $\Gm$-gerbe over $X$ corresponding to $n \alpha$,) which induces a map $\gamma$ on the coarse spaces. Since $n$ is the period of $\alpha$, we have $\gamma: \Aut^0_{\X} \to X \times X^{\vee}$. We then see that $\Phi^{00}(\wedge^n \E) = \gamma(\Phi^{00}(\E)) = \gamma(\Phi^{00}(\E')) = \Phi^{00}(\wedge^n \E')$. Indeed, $\Phi^{00}(\wedge^n \E) \subset \gamma(\Phi^{00}(\E))$ and they are proper schemes of the same dimension, similarly for $\E'$. By \cite[Lemma 6.8]{mukai1978semi}, this shows that $\delta(\wedge^n \E) = \delta(\wedge^n \E')$, which in turn shows that $\delta(\E) = \delta(\E')$ by the fact that $\delta(\wedge^n \E) = n \delta(\E)$.
    \end{enumerate}
\end{example}

\begin{corollary}
\label{relation}
    If $\E$ and $\E'$ are $1$-twisted simple semi-homogeneous vector bundles with $\delta_{\X}(\E) = \delta_{\X}({\E'})$, then if $\E \not \cong \E'$, $\Ext^s(\E, \E') = 0$ for any $s \in \mathbb{Z}$, and $\E \cong \E' \otimes \M$ for some $0$-twisted line bundle $\M$.
\end{corollary}

\begin{proof}
    By Corollary \ref{compare}, if $\E \neq \E'$, then $\Sigma(\E) \cap \Sigma(\E') = \varnothing$, so $\Ext^s_{\X}(\E, \E') =0$ for all $s$.

    For the second statement, consider $\mathcal{H} := \HHom_{\X}(\E, \E')$, $\mathcal{H}$ is $0$-twisted and homogeneous on $X$. By \cite[Proposition 4.18 (1)]{mukai1978semi}, 
    we see that there exists a line bundle $\M$ on $X$ (hence $0$-twisted on $\X$) such that $\Hom_{\X}(\E, \E' \otimes \M) \neq 0$. This implies that $\E \cong \E' \otimes \M$ by the first statement we showed.

    % And applying Theorem \ref{linebundle} to $\E$, we know that for some $\X' \rightarrow \X$ over an isogeny $f: X' \rightarrow X$ and line bundle $\N$ on $\X'$, \[ \Ext^s_{\X} (\E, \E) \cong \Ext^s_{\X'}(f_*{\N}, f_*{\N}) \cong \Ext^s_{\X'}(\oplus_{a \in \ker{f}}t_a^*{\N}, \N) \cong \Ext^s_{\X'}(\N, \N) \cong k^{\dim{X'}}.\]
    % The last isomorphism is by \cite[Theorem 5.8]{mukai1978semi} and the fact that $\X' \rightarrow X'$ is the trivial $\Gm$-gerbe.
\end{proof}

Consider a closed immersion $i : A \xhookrightarrow{} B$ of abelian varieties and $p: \mathscr{B} \rightarrow B$, a $\mathbb{G}_m$-gerbe over $B$. Let $\A := A \times_B \B$. We denote the base change of $p$, $\A \rightarrow A$ by $p$ as well.

\begin{lemma}
\label{decomp}
Let $\FF \in D(\mathscr{B})^{(1)}$ be a complex on $\mathscr{B}$ satisfies the following: 
\begin{itemize}
    \item for some $\sigma \in Aut(\mathscr{B})$, we have \[\sigma^* \FF \cong \FF\] if and only if there is some $a \in A(k)$ such that $p(\sigma) = i(a)$ where $p : Aut(\mathscr{B}) \rightarrow B$;
    \item $\mathcal{H}^s(\FF_{\sigma})$ is supported on $\A$.
\end{itemize}
Then 
\[ {\FF} \cong \oplus_{\sigma \in \Sigma({\FF})} \FF_{\sigma} \]
where \[\Sigma(\FF) : = \cup_s \Sigma(\G_s) \in G(k)\] and $\FF_{\sigma}$ has the property that $\mathcal{H}^s(\FF_{\sigma}) \cong \tilde{i}_* \G_s$ where $\tilde{i}$ is the base change of $i$ by $p$.
\end{lemma}

\begin{remark}
    One can think of this lemma as a generalization of Theorem \ref{summand}.
\end{remark}

\begin{proof}
Now we make sense of the notion of $\Sigma(\FF) : = \cup_s \Sigma(\G_s) \in G(k)$. Given a complex $F$ as in the Lemma, $\Sigma{(\mathcal{H}^i(F))}$ can be defined in the same set, that is, $G_i$ are homogenized by the same $(f, \mathcal{N})$.
Let $S_F \subset Aut^0_{\B}$ be the stablizer of $[F] \in \mathscr{D}_{\B}(k)$. The assumption on $\FF$ implies that $S_F \to B$ surjects onto $A \subset B$.

We then see that there is a group scheme homomorphism $S_F \to Aut^0_{\A}$ which surjects on to $A$ under $Aut^0_{\A} \to A$. This implies that $\Phi_{\A}^{00}(G_i) = \Phi^{00}(G_j)$ for all $G_i, G_j \neq 0$. Then the claim follows from the discussion in Example \ref{example}.

    The idea of the main proof is the same as \cite[Lemma 4.6]{de2022point}. Here we mimic the proof via induction. For $s$ big enough, for $\FF_{\geq s} = 0$ the conclusion is true. Now for inductive step, we look at the distinguished triangle 
    \[\begin{tikzcd}
	{\tilde{i}_*{\mathcal{G}_{s-1}}[-s+1]} & {\FF_{\geq s-1}} & {\FF_{\geq s}} & {\tilde{i}_*{\mathcal{G}_{s-1}}[-s+2]} \\
	{\bigoplus_{\sigma} \tilde{i}_*{\mathcal{G}_{s-1, \sigma}}[-s+1]} && {\bigoplus_{\sigma}\FF_{\geq{s,}\sigma}} & {\bigoplus_{\sigma} \tilde{i}_*{\mathcal{G}_{s-1, \sigma}}[-s+2]}
	\arrow[from=1-1, to=1-2]
	\arrow["\cong"', from=1-1, to=2-1]
	\arrow[from=1-2, to=1-3]
	\arrow[from=1-3, to=1-4]
	\arrow["\cong", from=1-3, to=2-3]
	\arrow["\cong", from=1-4, to=2-4]
\end{tikzcd}\]

To prove the induction step, it suffices to show that for $\sigma \neq \sigma'$, \[\Ext^*(\FF_{\geq s \sigma}, \tilde{i}_*{\mathcal{G}_{s-1, \sigma'})} = 0.\] Considering the canonical filtration of $\FF_{\geq s \sigma}$, this is true because for $s \neq t$ \[\Ext^*(\tilde{i}_*{\mathcal{G}_{t, \sigma}, \tilde{i}_*{\mathcal{G}_{s, \sigma'})}} \cong \Ext^*(L \tilde{i}^* \tilde{i}_*{\mathcal{G}_{t, \sigma}, \mathcal{G}_{s, \sigma'})} = 0\]

Indeed, the last equality is true by Lemma \ref{Ext} and the following lemma.

\end{proof}

%which shows that for $\mathcal{G}$ semi-homogeneous vector bundle on $\A$ for which $\Sigma(\mathcal{G})$ is defined, $\Sigma(L \tilde{i}^* i_* \mathcal{G}) = \Sigma(\mathcal{G})$.

\begin{lemma}
\label{sigma}
Let $\F$ be a semi-homogeneous vector bundle on $\mathscr{A}$. Then $\Sigma(L\tilde{i}^* \tilde{i}_* \F) = \Sigma(\F)$ where $\Sigma(L\tilde{i}^* \tilde{i}_* \F) := \bigcup \Sigma(\mathcal{H}^s(L\tilde{i}^* \tilde{i}_* \F))$.
\end{lemma}

\begin{proof}
    By the projection formula, we have \[\tilde{i}_* L\tilde{i}^* \tilde{i}_* \F \cong \tilde{i}_*(L \tilde{i}^* \tilde{i}_* \mathcal{O}_{\A} \otimes_{\mathcal{O}_{\A}}\F) \cong \tilde{i}_*(L \tilde{i}^* \tilde{i}_* p^*\mathcal{O}_{A} \otimes_{\mathcal{O}_{\A}}\F) \cong i_*(p^*Li^* i_* \mathcal{O}_A \otimes_{\mathcal{O}_{\A}} \F)\]
     As shown in \cite[Lemma 4.5]{de2022point}, $\mathcal{H}^s(Li^* i_* \mathcal{O}_A) \cong \underline{W_s} \otimes \mathcal{O}_A$, where  $\underline{W_s}$ the constant sheaf of some vector space $W_s$.

     So we have \[\tilde{i}_*\mathcal{H}^s(L \tilde{i}^* \tilde{i}_* \F) \cong \tilde{i}_*(\oplus_{\rk W_s} \F),\] hence \[\mathcal{H}^s(L \tilde{i}^* \tilde{i}_* \F) \cong \oplus_{\rk W_s} \F.\]

     This shows that $\Sigma(L\tilde{i}^* \tilde{i}_* \F) = \Sigma(\F)$.
\end{proof}

\begin{remark}
    This lemma can also be shown using \cite[Proposition 11.8]{huybrechts2006fourier} and projection formula.
\end{remark}

\subsection{Main Tool for Theorem \ref{point}}
    The main goal of this section is to prove the following proposition, which is crucial to conclude that point objects (objects in $\D(\X)^{(1)}$ that satisfies the conditions in Theorem \ref{point}) are indeed semi-homogeneous vector bundles on gerbes over the torsors of sub-abelian varieties of $X$.

  Again, consider a closed immersion $i : A \xhookrightarrow{} B$ of abelian varieties, $p: \B \rightarrow B$, a $\mathbb{G}_m$-gerbe over $B$, and its base change $p: \A \rightarrow A$.

\begin{proposition}
\label{data}
Let $\FF \in D(\mathscr{B})^{(1)}$ be a complex satisfying the following: 
\begin{enumerate}
    \item for some $\sigma \in Aut(\mathscr{B})$, we have \[\sigma^* \FF \cong \FF\] if and only if there is some $a \in A(k)$ such that $p(\sigma) = i(a)$ where $p : Aut(\mathscr{B}) \rightarrow B$.
    \item the complex $\FF$ is set-theoretically supported on $\mathscr{A}$;
    \item $\End(\FF) = k$;
    \item $\Ext^i(\FF, \FF) = 0$ for $i < 0$.
\end{enumerate}
Then there exists a line bundle $\LL$ such that $\FF \cong i_* \pi_* \LL[s]$ for an integer s, where $\pi$ is the map gotten by Theorem \ref{pushforward} applied to $X = A, \X = \mathscr{A}$.
\end{proposition}

\begin{proof}
    First we see that $\mathcal{H}^i(\FF)$'s are scheme-theoretically supported on $\A := \B \times_B A$. Consider the following diagram:
    \[\begin{tikzcd}
	{\mathscr{A}} & {\mathscr{X}} \\
	A & B & {B/A} \\
	&& U
	\arrow["{\tilde{i}}", from=1-1, to=1-2]
	\arrow[from=1-1, to=2-1]
	\arrow[from=1-2, to=2-2]
	\arrow[from=1-2, to=2-3]
	\arrow["i", from=2-1, to=2-2]
	\arrow[from=2-2, to=2-3]
	\arrow[hook, from=3-3, to=2-3]
\end{tikzcd}\]

Let $U := \Spec R$ be an open neighborhood of the origin in $B/A$, let $\B_U := \B \times_{B/A} U$. Since $\FF$ is set-theoretically supported on $\A$, we can view $\mathcal{H}^i(\FF)$ as on $\B_U$ (ie. it is isomorphic to the pushforward of its restriction to $\B_U$). To show $\mathcal{H}^i(\FF)$ is scheme-theoretically supported on $\A$, it suffices to show that the maximal ideal $m \subset R$ corresponding to the origin in $B/A$ annihilates $\mathcal{H}^i(\FF)$. This is true because of the fact that $R$ acts on $\mathcal{H}^i(\FF)$ via the following map:
\[R \rightarrow \Hom_{\B_U}(\FF|_{\B_U}, \FF|_{\B_U}) \cong \Hom_{\B}(\FF, \FF) = k\]

From the fact that $\End(\FF, \FF) = k$ and Lemma \ref{decomp}, we get $\mathcal{H}^s(\FF) \cong \tilde{i}_*{\mathscr{G}_s}$ for some $\mathscr{G}_s$ on $\A$ with $\Sigma(\mathscr{G}_s) = \{ \sigma \}$ for all $\mathcal{H}^s(\FF) \neq 0$. $\mathscr{G}_s$ is semi-homogeneous by condition 1.

Finally we show that $\FF$ is concentrated in only one degree. Suppose not, let $n$ be the top degree and $m$ be the bottom degree. We have that 
\[\Ext^{m-n}(\FF, \FF) \cong \Hom(\FF[n-m], \FF) \cong \Hom(\tilde{i}_*\G_n, \tilde{i}_*{\G}_m) \cong \Hom_{\A}(\G_n, \G_m)\] which is nonzero by Lemma \ref{Ext}. This contradicts the fact that $\Ext^i(\FF, \FF) = 0$ for all $i < 0$.

Therefore, we see that $\FF \cong \tilde{i}_*{\G}$ for some semi-homogeneous $\G$ on $\A$. Then we conclude by Theorem \ref{linebundle}.

\end{proof}

%\begin{lemma}
%Let $F^\bullet \in D(\mathscr{B})^{(1)}$ be a complex on $\mathscr{B}$ such that for every $a \in A(k)$, we have \[\sigma^* F^\bullet \cong F^\bullet\]
%for some $\sigma \in Aut(\mathscr{B})$ such that $p(\sigma) = i(a)$ where $p : Aut(\mathscr{B}) \rightarrow B$. Then 
%\[ {\FF} \cong \oplus_{\sigma \in \Sigma({\FF})} F_{\sigma}^\bullet \]
%where \[\Sigma(\FF) : = \cup_s \Sigma(\G_s) \in G(k)\] and $F_{\sigma}^\bullet$ has the property that $\mathcal{H}^s(F_{\sigma}^\bullet) \cong i_* \G_s$
%\end{lemma}

%Now we can write out the proof of Proposition \ref{data}.

\section{Classification of Point Objects}
In this section $\Char{k} = 0$, not necessarily algebraically closed. 

\begin{definition}
For any $k$-scheme $S$, we call a $S$-perfect complex $F_S \in \D(\X_S)^{(1)}$ that satisfies the conditions in Theorem \ref{point} at every geometric point $\bar{s} \rightarrow S$ a \textbf{relative point object}. 
\end{definition}

Recall from section 2 that we have an algebraic stack $\mathscr{D}_{\X}$ over $k$, which associates $S$ to the groupoid of $K \in \D(\X_S)^{(1)}$ such that for all geometric point $\bar{s} \rightarrow S$ we have $\Ext^i(K_s, K_s) = 0$ for $i < 0$. Let $\mathscr{P}$ be the fibered category over the category of $k$-schemes which to any $S$ associates the groupoid of relative point objects.

\begin{lemma}
    $\mathscr{P}$ is a locally closed substack of $\mathscr{D}_{\X}$, hence an algebraic stack.
\end{lemma}
\begin{proof}
    The condition that $\Hom(K_s, K_s) = k$ defines a locally closed substack of $\mathscr{D}_{\X}$ by \cite[\href{https://stacks.math.columbia.edu/tag/0BDL}{Tag 0BDL}]{stacks-project}, and by semi-continuity, the condition $\dim(\Ext^1(K_s, K_s)) \leq d$ is an open condition.
\end{proof}

\begin{remark}
    This in turn shows that $\mathscr{P}$ is a $\mathbb{G}_m$-gerbe over its coarse space $P$ since any object in $\mathscr{P}$ is universally simple.
\end{remark}

\begin{proof}[Proof of Theorem \ref{point}]
    Any $\FF \in \mathscr{P}$ defines a morphism of stacks \[\phi \colon \Autt^{0}_{\X} \rightarrow \mathscr{P}, \sigma \mapsto \sigma^* \FF.\]

Now, given any point object $\FF \in \D(\X)^{(1)}$, we show that it has the desired form. Note that it suffices to show the result over a algebraic closure $\bar{k}$ of $k$, so we assume $k = \bar{k}$.

Let $S_{\X} \subset \Aut^0_{\X}$ be the stablizer of the point $[\FF] \in P$ ($S_{\X}$ agrees with $\Phi(\FF)$ defined in earlier sections if $\FF$ is semi-homoegenous). The tangent space of $S_{\X}$ at the origin (the identity morphism) is the kernel of the morphism of tangent spaces induced by $\phi$. Since the tangent space of $\mathscr{P}$ at $\FF$ is $\dim \Ext^1(\FF, \FF) \leq d$ by assumption, we know that the tangent space of $S_{\X}$ as a k-vector space, hence $S_{\X}$ itself as an algebraic space, has dimension some integer $g = \dim {T_{\Aut^0_{\X, 0}}} - \dim T_{P_{F}} = 2d - \dim{\Ext^1(\FF, \FF)} \geq d$. 

Let $S$ be the image of $S_{\X}^0$ (i.e. the neutral connected component of $S_{\X}$) in $X$ under the map $\Aut_{\X} \rightarrow X$, and let $K$ be the kernel of $S_{\X}^0 \rightarrow S$. Note that $K \subset \Pic^0_X$. Let $g' := \dim{S}$.

The set-theoretic support of $\FF$ on $\X$ corresponds to a closed topological subspace $Z \subset X$, equip it with the reduced structure. We see that $Z$ is invariant under the $S$ action, which is free. Since $S$ is connected and normal hence irreducible, by picking a $k$-point in $Z$, we have embedding of $S$ into any irreducible component of $Z$, which would have dimension $\geq g'$.

Let $Z'$ be some irreducible component of $Z$ with $d' := \dim{Z'} \geq g'$.\\
We know that there is a smooth projective scheme $\tilde{Z}$ that surjects onto $Z'$, we get $\tilde{Z} \rightarrow X$, let $\tilde{T}$ be the albanese torsor of $\tilde{Z}$. Since $X$ is its own albanese torsor, we get $h \colon \tilde{T} \rightarrow X$ with $\dim{\im{h}} \geq d'$.

Taking duals, we get $h^\vee \colon X^\vee \rightarrow \tilde{T}^\vee$, and the dimension of the image of this map is $ \geq d'$, so $\dim{\ker{h^\vee}} \leq d-d'$.

Note also that the neutral connected component of $K$, $K^0 \subset \ker{h^\vee}$. This is because \[\FF \otimes \LL \cong \FF\]
and there exists some $\mathcal{H}^i({\FF})$ that is set-theoretically supported on the generic point of $Z$. So we have 
\[\mathcal{H}^i(\FF)|_{\tilde{Z}} \otimes \LL|_{\tilde{Z}} \cong \mathcal{H}^i(\FF)|_{\tilde{Z}}.\]

Let $\mathcal{H}$ be the quotient of $\mathcal{H}^i({\FF})|_{\tilde{Z}}$ by its torsion subsheaf. Then on the maximal open subset $U \subset \tilde{Z}$ such that $\mathcal{H}|_U$ is locally free (note that the complement of $U \subset \tilde{Z}$ has codimension at least 2,) we have that \[\det{\mathcal{H}|_{U}} \otimes \LL^{\otimes r}|_{U} \cong \det{\mathcal{H}|_{U}}\] which implies that $\LL|_{\tilde{Z}}$ is torsion, hence $\LL|_{\tilde{T}}$ is $r$-torsion. 

$K^0 \rightarrow \tilde{T}^\vee$ is continuous, since $K^0$ is connected, its image must only contain the structure sheaf.

So $\dim(K) \leq d-d'$ hence \[g = \dim(S) + \dim(K) \leq g' + (d-d') \leq d' +(d-d') = d\] with equality if and only if $g' = d'$ and $\dim(K) = d-d'$.
Using the inequality $g \geq d$, we conclude that $d' = g'$ and $Z'$ is a $S$ orbit. Since $\Hom(\FF, \FF) = k$ implies $\FF$ has connected support, $Z$ is connected, hence $Z = Z'$ and $Z$ is a torsor under $S$.

Now, we conclude the forward direction by Proposition \ref{data}.

\begin{remark}
\label{Ext^1=d}
    Note that the equality $g = d$ implies that whenever we have a point object $\F$, we necessarily have $\dim \Ext^1_{\X}(\F, \F) = d$ (instead of just $\dim \Ext^1_{\X}(\F, \F) \leq d$).
\end{remark}

Conversely, for objects of the form $\FF \cong \tilde{i}_*{\F}[s]$, we show that $\FF$ satisfies the conditions in Theorem \ref{point}.. By Theorem \ref{linebundle}, $\F := \tilde{\pi}_* \LL$ for the base change of some isogeny $\pi: Z' \rightarrow Z$ under $\Z \rightarrow Z$ (so $\FF \cong \tilde{i}_* \tilde{\pi}_* \LL[s]$.) We may assume that $s = 0$. The fact that $\Ext^m(\FF, \FF) = 0$ for $s < 0$ follows from the fact that $\FF$ is concentrated in one degree. 

For the other two conditions, first notice that \[\Ext^m_{\Z}(\tilde{\pi}_* \LL, \tilde{\pi}_* \LL) \cong \Ext^m_{\Z'}(\tilde{\pi}^* \tilde{\pi}_* \LL, \LL) \cong \Ext^m_{\Z'}(\oplus_{a \in (\ker{\pi} \cap Z)} a \cdot \LL, \LL) \cong \Ext_{\Z'}^m(\LL, \LL).\]
The last isomorphism follows from Proposition \ref{inject}.
So we have
\[ \Hom_{\X}(\FF, \FF) \cong \Hom_{\X}(\tilde{i}_* \pi_* \LL[s], \tilde{i}_* \tilde{\pi}_* \LL[s]) \cong \Hom_{\Z}(\tilde{\pi}_* \LL, \tilde{\pi}_* \LL) \cong \Hom_{\Z'}(\LL, \LL) = k.\]
And \[\Ext^1_{\Z}(\tilde{\pi}_* \LL, \tilde{\pi}_* \LL) \cong \Ext_{\Z'}^1(\LL, \LL).\] By \cite{mukai1978semi}, $\Ext^1_{\Z}(\F, \F) \cong \Ext^1(\LL, \LL)$ has dimension $\dim Z' = \dim Z$.

Now, consider the distinguished triangle 

\[\mathcal{H}^{-1}(L \tilde{i}^* \tilde{i}_* \FF)[1] \rightarrow \tau_{\geq -1}(L \tilde{i}^* \tilde{i}_* \FF) \rightarrow \F.\]

We see that $\tilde{i}_* \mathcal{H}^{-1}(L \tilde{i}^* \tilde{i}_* \FF) \cong \mathcal{H}^{-1}(\tilde{i}_* L \tilde{i}^* \tilde{i}_* \FF) \cong \mathcal{H}^{-1}(\tilde{i}_* L \tilde{i}^* \tilde{i}_*\mathcal{O}_Z \otimes \F)$, so \[\mathcal{H}^{-1}(L \tilde{i}^* \tilde{i}_* \FF) \cong \mathcal{H}^{-1}(L \tilde{i}^* \tilde{i}_* \mathcal{O}_Z) \otimes \F \cong \F^{\oplus \codim Z}.\]
The last isomorphism follows by Lemma \ref{sigma}.

So applying $\Hom{(-, \FF)}$ to the distinguished triangle, we get the exact sequence \[0 \rightarrow \Ext_{\Z}^1(\F, \F) \rightarrow \Ext_{\X}(\FF, \FF) \rightarrow \Hom_{\Z}(\FF^{\oplus \codim Z}, \F)\]

which implies that $\dim{\Ext_{\X}(\FF, \FF)} \leq d$ where $d = \dim X$.
\end{proof}

\section{Twisted Fourier-Mukai Partners with Abelian Varieties}
One of the main application of classifying point objects is to study Fourier-Mukai partners. In our case, it shows that any twisted Fourier-Mukai partner of an abelian variety is also be an abelian variety of the same dimension.

\subsection{Moduli of Point Objects}

Let $\mathscr{M}$ be the fibered category over $k$ which to any scheme $T$ associates the groupoid of pairs $(Z, \E)$, where $Z \subset X_T$ is a closed subscheme flat over $T$ and $\E$ is a $1$-twisted vector bundle on $\Z := \mathscr{X} \times_X Z$ such that for every geometric point $\bar{t} \hookrightarrow T$, $Z_{\bar{t}} \subset X_{\bar{t}}$ is a torsor under a subabelian variety $S_{\bar t} \subset A_{\bar t}$, and $\E_{\bar t}$ is a simple semi-homogeneous vector bundle on $\Z_t$.

\begin{lemma}
    Let $A$ be an abelian variety, let $Z \xhookrightarrow{} A$ be a closed subscheme of $A$ which is a torsor under some subabelian variety $S \xhookrightarrow{} A$. Let $\rho_S : A/S \rightarrow \Hilb_A$ be the map $\rho_S([a]) = [t_a^*(S)]$, this map is well-defined since $Z$ is a torsor under $S$. Then $\rho_S$ is an open and closed embedding.
\end{lemma}

\begin{proof}
    We see that since $S$ is precisely the stablizer of $[Z]$, so $\rho_S$ is a monomorphism on $k$-points. To show it is an embedding, we show that the map induced on the tangent space $T_{\rho_{S}}$ is an isomorphism. This is enough since by \cite[19.1.1]{vakil2025rising}, we know this implies that $\rho_S$ is a closed embedding (both $A/S$ and $\Hilb_A$ are projective). And isomorphism on the tangent spaces would imply that the component of $\Hilb_A$ that contains the image of $\rho_S$ has dimension $\leq \dim A/S$, which then implies that $\rho_S$ is an isomorphism.
    
    By the normal sequence
    \[0 \rightarrow \mathcal{T}_Z \rightarrow \mathcal{T}_{A}|_Z \rightarrow \N_{Z/A} \rightarrow 0\]
    where $\N_{S/A}$ is the normal bundle of $S \subset A$.

    Since we also have that the tangent space of $\Hilb_A$ at $[Z]$ is $H^0(Z, \N_{Z/A})$ and the tangent space of $A/S$ at every $k$-point is $H^0(Z, \mathcal{T}_A|_Z) / H^0(Z, \mathcal{T}_Z)$.  

    So taking the long exact sequence of the normal sequence, we see that the map $H^0(Z, \mathcal{T}_A|_Z) / H^0(Z, \mathcal{T}_Z) \rightarrow H^0(Z, \N_{Z/A})$ is the map on the tangent spaces induced by $\rho$. By the long exact sequence we also see that, to conclude, it suffices to show \[H^1(Z, \mathcal{T}_Z) \rightarrow H^1(Z, \mathcal{T}_A|_Z)\] is injective, which can be seen by noticing this is the map gotten by applying $- \otimes_k H^1(Z, \mathcal{O}_Z)$ to the injection $T_x(Z) \xhookrightarrow{} T_x(A)$.
\end{proof}

\begin{remark}
    This shows that for each connected component of $\mathscr{M}$, there exists a subabelian variety $S \rightarrow A$ and $\delta \in NS(S)_{\mathbb{Q}}$ such that for any pairs $(Z, \E) \in \mathscr{M}(k)$ we have $Z$ is a torsor under $S$ and $\delta(\E) = \delta$. In other words, we have that \[\mathscr{M} = \coprod_{(S, \delta)} \mathscr{M}_{(S, \delta)}.\]
    
    %Different components correspond to different such pairs $(S, \delta)$. So, for the connected component that corresponds to $(S, \delta)$, we call it $\mathscr{M}_{(S, \delta)}$.
\end{remark}

% \begin{lemma}
%     For any simple $\E$ and $\E'$ with $\delta(\E) = \delta(\E')$, there exists some $\mathcal{M} \in S^\vee$ such that $\E \cong \E' \otimes \mathcal{M}$.
% \end{lemma}
% \begin{proof}
%     Take some $f : S' \rightarrow S$ that trivialized $\mathscr{S} := \mathscr{A} \times_A S$, we see that $\delta_{S'}(\tilde{f}^*\E) = \delta_{S'}(\tilde{f}^*\E')$ and \cite[Lemma 6.17]{mukai1978semi} shows that $\tilde{f}^*\E \cong \tilde{f}^* \E' \otimes \LL'$ for some $\LL' \in S'^\vee$. Note that there is some $\LL \in S^\vee$ such that $\LL' \cong f^* \LL$. So we have that 
%     \[\tilde{f}^*\E \cong \tilde{f}^* \E' \otimes \tilde{f}^*\LL\]
%     applying $\tilde{f}_*$ and projection formula we get 
%     \[\bigoplus_{\N \in \ker{f^\vee}} \E \otimes \N \cong \bigoplus_{\N \in \ker{f^\vee}} \E \otimes \N \otimes \LL\] 
%     By Krull Schmitz theorem, we have that $\E \cong \E' \otimes \N_i \otimes \LL$ for some $\N_i \in \ker{f^\vee}$. Let $\mathcal{M}:= \N_i \otimes \LL$, we are done.
% \end{proof}

\begin{proposition}
    $\mathscr{M}_{(S, \delta)}$ is an algebraic stack locally of finite presentation over $k$.
\end{proposition}

\begin{proof}
    By the discussion above, we see that there is a map $\mathscr{M}_{(S, \delta)} \rightarrow A/S$. Let $Y$ be the universal object of $A/S$, which is, a closed subscheme of $A \times_k A/S$. Let $\Y = \A \times_A Y$. 
    
    Let $\mathscr{V}$ be the fibered category over $A/S$ which to any $T \rightarrow A/S$ associates the groupoid of simple $1$-twisted vector bundles $\E$ on $\Y_T$ such that $\delta(\E_t) = \delta$ for all geometric points $t \rightarrow T$. By \cite[Proposition 2.3.1.1]{lieblich2007moduli} and \cite[\href{https://stacks.math.columbia.edu/tag/0BDL}{Tag 0BDL}]{stacks-project}, $\mathscr{V}$ is an algebraic stack over $A/S$, hence an algebraic stack locally of finite presentation over $k$.

    We see that $\mathscr{M}_{(S, \delta)} \rightarrow \mathscr{V}$ is a monomorphism over $A/S$. By Theorem \ref{point}, a simple $1$-twisted vector bundle on $S_t$ is semi-homogeneous if and only if $\dim\Ext^1(\E, \E) \leq \dim S_t$. Semi-continuity implies that $\mathscr{M}_{(S, \delta)}$ is an open substack of $\mathscr{V}$, which then is an algebraic stack over $k$.
\end{proof}

\begin{proposition}
    Assume $k = \bar{k}$. When there exists a simple vector bundle $\E$ such that $\delta(\E) = \delta$ on $S$, the stack $\mathscr{M}_{(S, \delta)}$ is a $\mathbb{G}_m$-gerbe over an abelian variety $M_{(S, \delta)}$ which is an extension
    \[0 \rightarrow S^\vee/\Phi_0 \xrightarrow{d_1} M_{(S, \delta)} \xrightarrow{d_2} A/S \rightarrow 0\]
    where $\Phi_0 := \Phi_0(\E)$ a finite subgroup scheme of $S^\vee$.
\end{proposition}

\begin{proof}

    % Let $M_{(S, \delta)}$ be the sheafification of the functor from $k$-schemes to set which assigns to each $T/k$ the set of isomorphism class of pairs $(Z, \E)$, where $Z \subset A_T$ is a closed subscheme flat over $T$ and $\E$ is a vector bundle on $\Z := \mathscr{A} \times_X Z$ such that for every geometric point $\bar{t} \hookrightarrow T$, $Z_{\bar{t}} \subset A_{\bar{t}}$ is a torsor under the subabelian variety $S_{\bar t} \subset A_{\bar t}$, and $\E_{\bar t}$ is a simple semi-homogeneous vector bundle on $\Z$ which $\delta(\E_{\bar{t}}) = \delta$.

    % We see that $\mathscr{M}_{(S, \delta)} \rightarrow M_{(S, \delta)}$ is a $\Gm$-gerbe. Now we show that $M_{(S, \delta)}$ has the desired exact sequence.
    We first see that $\E$ induces a morphism $\AAut^0_{\A} \to \M_{(S, \delta)}$ that sends $\sigma$ to $\sigma^* \E$. The induced map on the coarse spaces is surjective. By \cite[Lemma 1.13]{lane2024semi} and the fact that this map is surjective on $k-$points, we see that $M_{(S, \delta), red}$ is a torsor under some quotient of $\Aut^0_{\A}$.

    % Indeed, this can be checked locally. Given any $T \to M_{S, \delta}$, it corresponds to a vector bundle on $\Y$ (defined in the proof of last proposition) whose geometric fibers are all semi-homogeneous vector bundles. Take the etale cover $\{U_i \to T\}$ such that $Y_{U_i} \cong S \times U_i$, then $U_i \to M_{(S, \delta)})$ is $\F$ a family $1-$ twisted semi-homogeneous vector bundles on $Z \times_A \A$ for some $S$-torsor $Z \subset A$. We see that there exists a $\sigma \in \Aut^0_{\cA}(k)$ such that $\sigma^* \E$ is supported on $Z$. Replacing $\E$ by $\sigma^* \E$, take $\M = \HHom(\sigma^* \E_{U_i}, \F_{U_i})$, which is a line bundle since its locally free and its fibers have dimension 1. Let $U_i \to X^{\vee} \to \Aut^0_{\cA}$ be the map given by $\M$, we see that $U_i \to \Aut^0_{S \times \cA} \to M_{(S, \delta)}$ is then the class of $\F$.

     We then equip $M_{(S, \delta), red}$ with the group structure by choosing the identity to be the point corresponding to $\E$. Now we show it is an abelian variety of dimension $\dim A$.
    
    Corollary \ref{relation} shows that $\Phi_0(\E) = \Phi_0(\E')$ for $\delta(\E) = \delta(\E')$, which makes $\Phi_0$ a well-defined notion.

     % Over any $T$, $\tilde{d_2}: \mathscr{M}_{(S, \delta)} \rightarrow A/S$ maps $(Z, \E)$ to $Z$ which is surjective (locally it is an epimorphism by assumption). 
    We now show we have
     \[0 \rightarrow S^\vee/\Phi_0 \xrightarrow{d_{1,red}} M_{(S, \delta), red} \xrightarrow{d_{2,red}} A/S \rightarrow 0.\]
     
     Surjectivity of $d_{2, red}$ can again be checked on closed points. We can see that the kernel of $d_2$ is a closed subgroup of $M_{(S, \delta)}$. We know that $\ker{d_2}(k) = \{\text{semi-homogeneous vector bundles} \E \text { with } \delta(\E) = \delta \text{ on }S\}$ and that $\ker{d_2}$ is reduced since it is a group scheme in characteristic 0.

    We see that $S^{\vee}/\Phi_0$ acts on $\ker{d_2}$ by tensoring, this action is transitive on $k-$points by Corollary \ref{relation}, which shows that $\ker{d_2}$ is a torsor under $S^{\vee}/\Phi_0$ again by \cite[Lemma 1.13]{lane2024semi}.

    We then see that $\dim M_{(S, \delta)} = \dim M_{(S, \delta), red} = \dim A$. Now, by Remark \ref{Ext^1=d}, we see that the tangent space of $M_{(S, \delta)}$ has dimension $d$. Regularity implies reducedness, we then conclude by noting $M_{(S, \delta)} = M_{(S, \delta), red}$.

    %$d_1$ is defined by the following: over any $T$, we can send $\L \in S^{\vee}_T$ to the pair $(S_T, \E_T \otimes \L)$ which is an object in $\mathscr{M}_{(S, \delta)}$,we get $d_1$ by composing this map with the morphism $\mathscr{M}_{(S, \delta)} \rightarrow M_{(S, \delta)}$.
\end{proof}

\begin{corollary}
    In general, every connected component of $\mathscr{M}$, $\mathscr{M}_{(S, \delta)}$ is a $\Gm$-gerbe over $M_{(S, \delta)}$, a torsor under an abelian variety.
\end{corollary}

\begin{remark}
    Note that this proposition shows, in particular, that $\dim(M_{(B, \delta)}) = \dim A$.
\end{remark}

\subsection{Twisted FM Partners}
The universal object on $\X \times_k \mathscr{M}_{(S, \delta)}$ is $(1,1)$-twisted, hence induces a functor $\Psi:  \D(\mathscr{M}_{(S, \delta)})^{(-1)} \rightarrow \D(\mathscr{X})^{(1)}$.

\begin{remark}
    By the discussion below \cite[Definition 2.1.2.2]{lieblich2005moduli}, for any $X$, the category of $1$-twisted coherent sheaves on some $\Gm$-gerbe over $X$ corresponding to $\alpha \in \Br(X)$ is equivalent to the $-1$-twisted coherent sheaves on the $\Gm$-gerbe over $X$ corresponding to $\alpha^{-1}$.
\end{remark}

\begin{proposition}
\label{equivalence}
    $\Psi$ is an equivalence.
\end{proposition}

\begin{proof}
    It suffices to show the statement over $\bar{k}$, without loss of generality, we assume $k = \bar{k}$, and therefore $X \cong A$ as abelian varieties.
    
    By Example \ref{example}, we see that for any simple semi-homogeneous $\E$ and $\E'$ with $\delta(\E) = \delta(\E')$ on some $S$ torsor $Z \subset A$, we can define $\Sigma(\E)$ and $\Sigma(\E')$ via the same isogeny $\pi: S' \rightarrow S$. Hence Lemma \ref{Ext} applies. By deformation theory and Lemma \ref{Ext}, \cite[Theorem A. 40]{lane2024semi} implies that $\Psi$ is fully faithful (since $\Psi(k(x) \otimes \chi_M^{-1})$ is such a point object $\E$.) 
    
    By the same proof as in \cite[Remarks 3.37 (ii)]{huybrechts2006fourier}, we see that $[\dim A]$ shifting by the dimension of $A$ defines a Serre funtor in both $\D(\mathscr{A})^{(1)}$ and $\D(\mathscr{M}_{(S, \delta)})^{(-1)}$, and our $\Psi$ respects the two Serre functors. Hence by \cite[Corollary 1.56]{huybrechts2006fourier}, $\Psi$ is an equivalence.
\end{proof}

\begin{remark}
\label{embeddding of M}
    We have a natural map $\tilde{i}_M: \M_{(S, \delta)} \to \mathscr{P}$ where $\mathscr{P}$ is the moduli of weak point objects defined in the last section. Similarly as in \cite[Lemma 5.2]{lieblich2015fourier}, we see that the map induced on the coarse space $i_M : M_{(S, \delta)} \to P$ is an open embedding. 
    % Note that since $\tilde{i}_M$ is banded, $\tilde{i}_M$ is the base change of $i_M$ via $\mathscr{P} \to P$.
\end{remark}

\begin{theorem}
\label{partner}
    Suppose $Y$ is a smooth projective variety over $k$, $\Y$ a $\mathbb{G}_m-$gerbe over $Y$ corresponding to $\beta \in \Br(Y)$, $A/k$ an abelian variety, $X/k$ a torsor under $A$, and $\X$ a $\mathbb{G}_m-$gerbe over $X$ corresponding to $\alpha \in \Br(X)$.  If there is an equivalence $\Gamma: \D(Y, \beta) \cong \D(X, \alpha)$, then $Y \cong M_{(S, \delta)}$ for some subabelian variety $S \subset A$ and $\delta \in NS(A) \otimes \mathbb{Q}$. In particular, $Y$ is a torsor under an abelian variety of dimension $\dim{A}$.
\end{theorem}

\begin{proof}
    Let $P \in \D(Y \times X, (\beta, \alpha))$ be the object defining the equivalence $\Gamma$. By \cite[Theorem 1.1]{canonaco2007twisted} we know such $P$ exists. Up to a shift, we may assume that the smallest $i$ such that $\mathcal{H}^i(P) \neq 0$ is $0$. 

    \begin{lemma}
        $P$ is a sheaf flat over $\Y$.
    \end{lemma}

    \begin{proof}
        We first show that $P$ is a sheaf, it suffices to show this over $\bar{k}$ so we assume $k = \bar{k}$. For any $y: \Spec{k} \rightarrow \Y$, let $k(y) \otimes \chi_Y$ be the $1$-twisted skyscraper sheaf at $y$. Since we assume $P$ is concentrated in non-negative degrees, we have $\mathcal{H}^{0}(\Gamma(k(y) \otimes \chi_Y)) \neq 0$. Indeed, we may replace $P$ by a locally free complex with the same cohomology sheaves, then $\mathcal{H}^{0}(\Gamma(k(y) \otimes \chi_Y)) \cong \mathcal{H}^0(P_y) \cong \mathcal{H}^0(P)_y$. Since $\mathcal{H}^0(P)$ is a subsheaf of a locally free sheaf, it is torsion free, which imples that $\mathcal{H}^0(P)_y \neq 0$. We also know that $\Gamma(k(y) \otimes \chi_Y)$ must be a point object on $\X$, which then must be a sheaf concentrated in degree zero. 
        
        Now, as in \cite[Lemma 3.31]{huybrechts2006fourier}, replacing $S$ there by $\X \times \Y$, $X$ there by $\Y$, the functor $i^*$ by $- \otimes (k(y) \otimes \chi_Y)$, and spectral sequence \cite[(3.10)]{huybrechts2006fourier} by spectral sequence \cite[(3.9)]{huybrechts2006fourier}, the same proof shows that $P$ must also be concentrated in degree zero (ie. $P$ is a sheaf), and that $P$ is flat over $\Y$.

        % We claim that $P$ is locally free on its scheme-theoretic support. This can be checked locally so we may assume that $\alpha$ and $\beta$ are zero, that is, we may replace $\X$ by $X$ and $\Y$ by Y. Then we have that at any closed point $y \in Y$, $P_{X_y}$ is locally free, which would imply that $P$ is locally free on the scheme-theoretic support of $P$. Since $P$ is flat over $\Y$, this implies that its scheme-theoretic support is flat over $\Y$.
    \end{proof}

    Let $\Y^{-1}$ be the $\Gm$-gerbe corresponding to $\beta^{-1} \in \Br(Y)$. We have that $P$ defines a banded morphism $\tilde{i}_Y: \Y^{-1} \to \mathscr{P}$ such that the induced morphism $i_Y:Y \to P$ is an open embedding.
    By Theorem \ref{point} and the fact that $P$ is a sheaf (concentrated in degree 0), we see that $i_Y$ factors through $i_M$. Since both $\tilde{i}_Y$ and $\tilde{i}_M$ are banded, we see that $\tilde{i}_Y$ factors through $\tilde{i}_M$, that is, $P$ defines a morphism $\tilde{g}': \Y^{-1} \rightarrow \mathscr{M}$. Since $\Y^{-1}$ is connected, \[\tilde{g}': \Y^{-1} \rightarrow \mathscr{M}_{(S, \delta)}\] for some $(S, \delta)$. Note that this is still a morphism of gerbes (since the pullback of $(1,1)$-twisted universal bundle is still $(1,1)$-twisted on $\X \times \Y^{-1}$).

    Let $g: Y \rightarrow M_{(S, \delta)}$ be the map induced on coarse spaces by $\tilde{g}'$. Let $\mathscr{M}^{-1}_{(S, \delta)} \rightarrow M_{(S, \delta)}$ be the $\Gm$-gerbe corresponding to $\gamma^{-1} \in \Br(M_{(S, \delta)})$, where $\gamma$ corresponds to $\mathscr{M}_{(S, \delta)}$. Then the base change of $g$ by $\mathscr{M}^{-1}_{(S, \delta)} \rightarrow M_{(S, \delta)}$ gives a morphism of gerbes \[\tilde{g}: \Y \rightarrow \mathscr{M}^{-1}_{(S, \delta)}.\]
    
    Considering the functor $\tilde{g}^*: \D(M_{(S, \delta)}, \gamma^{-1}) \rightarrow \D(Y, \beta)$ induced by $\tilde{g}$, we see that $\Psi = \Gamma \circ \tilde{g}^*$. 

    This implies that $i_Y = i_M \circ g$, so we know that $g$ is also an open embedding. Since both the source and target of $g$ are proper, connected, and of the same dimension, we see that $g$ is an isomorphism.

    % So $\tilde{g}^*$ is an equivalence. Passing to $\bar{k}$ we see that over all closed points, the fibers of $g_{\bar{k}}$ are non-empty, $0$-dimensional, connected, and reduced, therefore $g_{\bar{k}}$ is a surjective embedding, hence an isomorphism. This is enough to conclude that $g$ is an isomorphism.
\end{proof}

\begin{corollary}
    Under the same assumptions, we have the following commutative diagram where the top morphism is a morphism of gerbes: 
    \[\begin{tikzcd}
	\Y & {\mathscr{M}^{-1}_{(S, \delta)}} \\
	Y & {M_{(S,\delta)}}
	\arrow["\cong", from=1-1, to=1-2]
	\arrow[from=1-1, to=2-1]
	\arrow[from=1-2, to=2-2]
	\arrow["\cong", from=2-1, to=2-2]
    \end{tikzcd}\]
    That is, the class $\beta \in \Br(Y)$ is determined by the class of $\mathscr{M}_{(S, \delta)} \rightarrow M_{(S, \delta)}$.
\end{corollary}

\begin{remark}
    Let $\mathscr{P}_{\sigma}$ be a connected component of $\mathscr{P}$ defined in Section 5. Using techniques as in Proposition \ref{equivalence}, one can show that $\D(\mathscr{P}_{\sigma})^{(-1)} \cong \D(\A)^{(1)}$. This shows that $\mathscr{M}_{(S, \delta)}$'s are isomorphic to the connected components of $\mathscr{P}$.

    Indeed, one can define the moduli of point objects to be the fibered category over $k$ which to any $T$ associates the groupoid of relative point objects. Our $\mathscr{M}$ is the open substack of $\mathscr{P}$ whose objects are complexes concentrated in degree $0$.
\end{remark}

\nocite{*}
\printbibliography[heading=bibintoc]

@article{de2022point,
  title={Point Objects on Abelian Varieties},
  author={de Jong, Aise Johan and Olsson, Martin},
  journal={arXiv preprint arXiv:2209.05553},
  year={2022}
}

@article{atiyah1956krull,
  title={On the Krull-Schmidt theorem with application to sheaves},
  author={Atiyah, Michael F},
  journal={Bulletin de la Soci{\'e}t{\'e} math{\'e}matique de France},
  volume={84},
  pages={307--317},
  year={1956}
}

@article{mukai1978semi,
  title={Semi-homogeneous vector bundles on an Abelian variety},
  author={Mukai, Shigeru},
  journal={Journal of Mathematics of Kyoto University},
  volume={18},
  number={2},
  pages={239--272},
  year={1978},
  publisher={Duke University Press}
}

@book{EGAIII,
  author    = {A. Grothendieck and J. Dieudonn{\'e}},
  title     = {{\'E}l{\'e}ments de g{\'e}om{\'e}trie alg{\'e}brique. III: {\'E}tude cohomologique des faisceaux coh{\'e}rents},
  series    = {Publications Math{\'e}matiques de l’IH{\'E}S},
  volume    = {11},
  year      = {1961},
  publisher = {Institut des Hautes {\'E}tudes Scientifiques},
  address   = {Bures-sur-Yvette},
  note      = {EGA III},
  url       = {http://www.numdam.org/item/PMIHES_1961__11__5_0/}
}

@misc{stacks-project,
  author       = {The {Stacks project authors}},
  title        = {The Stacks project},
  howpublished = {\url{https://stacks.math.columbia.edu}},
  year         = {2025},
}

@phdthesis{caldararu2000derived,
  title={Derived Categories of Twisted Sheaves on Calabi-Yau Manifolds},
  author={Andrei Căldăraru},
  year={2000},
  school={Cornell University},
  url={https://api.semanticscholar.org/CorpusID:115714156}
}

@article{lieblich2007moduli,
  title={Moduli of twisted sheaves},
  author={Max Lieblich},
  journal={Duke Mathematical Journal},
  year={2004},
  volume={138},
  pages={23-118},
  url={https://api.semanticscholar.org/CorpusID:14067307}
}

@article{lieblich2005moduli,
  title={Moduli of complexes on a proper morphism},
  author={Max Lieblich},
  journal={Journal of Algebraic Geometry},
  year={2005},
  volume={15},
  pages={175-206},
  url={https://api.semanticscholar.org/CorpusID:10538350}
}

@article{olsson2025twisted,
  title={Twisted derived categories and Rouquier functors},
  author={Olsson, Martin},
  journal={International Mathematics Research Notices},
  volume={2025},
  number={18},
  pages={rnaf290},
  year={2025},
  publisher={Oxford University Press}
}

@book{olsson2023algebraic,
  title={Algebraic spaces and stacks},
  author={Olsson, Martin},
  volume={62},
  year={2023},
  publisher={American Mathematical Society}
}

@article{canonaco2007twisted,
  title={Twisted Fourier--Mukai functors},
  author={Canonaco, Alberto and Stellari, Paolo},
  journal={Advances in mathematics},
  volume={212},
  number={2},
  pages={484--503},
  year={2007},
  publisher={Elsevier}
}

@article{lane2024semi,
  title={Semi-Homogeneous Sheaves and Twisted Derived Categories},
  author={Lane, Tyler},
  journal={arXiv preprint arXiv:2411.10274},
  year={2024}
}

@book{huybrechts2006fourier,
  title={Fourier-Mukai transforms in algebraic geometry},
  author={Huybrechts, Daniel},
  year={2006},
  publisher={Clarendon Press}
}

@book{huybrechts2010geometry,
  title={The geometry of moduli spaces of sheaves},
  author={Huybrechts, Daniel and Lehn, Manfred},
  year={2010},
  publisher={Cambridge University Press}
}

@article{bergh2021decompositions,
  title={Decompositions of derived categories of gerbes and of families of Brauer-Severi varieties},
  author={Bergh, Daniel and Schn{\"u}rer, Olaf M},
  journal={Documenta Mathematica},
  volume={26},
  pages={1465--1500},
  year={2021}
}

@article{kurama2024fourier,
  title={Fourier-Mukai partners of abelian varieties and K3 surfaces in positive and mixed characteristics},
  author={Kurama, Riku},
  journal={arXiv preprint arXiv:2410.14065},
  year={2024}
}

@article{krashen2008index,
  title={Index reduction for Brauer classes via stable sheaves},
  author={Krashen, Daniel and Lieblich, Max},
  journal={International Mathematics Research Notices},
  volume={2008},
  number={9},
  pages={rnn010--rnn010},
  year={2008},
  publisher={OUP}
}

@inproceedings{lieblich2015fourier,
  title={Fourier-Mukai partners of K3 surfaces in positive characteristic},
  author={Lieblich, Max and Olsson, Martin},
  booktitle={Annales scientifiques de l'{\'E}cole Normale Sup{\'e}rieure},
  volume={48},
  number={5},
  pages={1001--1033},
  year={2015}
}

@article{vakil2025rising,
  title={The rising sea: Foundations of algebraic geometry},
  author={Vakil, Ravi},
  year={2025},
  publisher={Princeton University Press}
}
\end{document}